\documentclass[12pt,CJK]{amsart}

\usepackage[margin=1.5in]{geometry}
\usepackage{amscd,amssymb, amsmath,amsfonts, wasysym, mathrsfs, enumerate, mathtools,hhline,color}
\usepackage{CJK}
\usepackage{graphicx}
\usepackage[all, cmtip]{xy}
\usepackage{xcolor}

\usepackage{url}

\definecolor{hot}{RGB}{65,105,225}

\usepackage[pagebackref=true,colorlinks=true, linkcolor=hot ,  citecolor=hot, urlcolor=hot]{hyperref}

\theoremstyle{plain}
\newtheorem{theorem}{Theorem}[section]
\newtheorem{proposition}[theorem]{Proposition}

\newtheorem{corollary}[theorem]{Corollary}

\newtheorem{lemma}[theorem]{Lemma}

\theoremstyle{definition}
\newtheorem{definition}[theorem]{Definition}
\newtheorem{remark}[theorem]{Remark}

\newtheorem{example}[theorem]{Example}

\newtheorem{bigthm}{Theorem}
\newtheorem{bigcor}[bigthm]{Corollary}

\newtheorem*{acks}{Acknowledgements}

\newtheorem*{out}{Outline}

\numberwithin{equation}{section}

\def\:{\colon\!}

\title[Poincar\'{e} Pairs]{On the various notions of Poincar\'{e} duality pair}

\begin{document}
\author{John R. Klein}
\address{Department of Mathematics, Wayne State University, Detroit, MI 48202, USA}
\email{klein@math.wayne.edu}

\author{Lizhen Qin}
\address{Department of Mathematics, Nanjing University, Nanjing, Jiangsu 210093, P.R.China}
\email{qinlz@nju.edu.cn}

\author{Yang Su}
\address{HLM, Academy of Mathematics and Systems Science, Chinese Academy of Sciences, Beijing 100190, China}
\address{School of Mathematical Sciences, University of Chinese Academy of Sciences, Beijing 100049, China}
\email{suyang@math.ac.cn}

\maketitle

\begin{abstract}
We establish a number of foundational results on Poincar\'{e} spaces which result in several applications.
One  application settles an old conjecture of C.T.C. Wall in the affirmative.
Another result shows that for any natural number $n$, there exists a finite CW pair $(X,Y)$
satisfying relative Poincar\'e duality in dimension $n$ with the property that $Y$ fails to satisfy Poincar\'e duality.
We also prove a relative version of a result of Gottlieb about Poincar\'e duality and fibrations.

\end{abstract}

\setcounter{tocdepth}{1}
\tableofcontents
\addcontentsline{file}{sec_unit}{entry}

\section{Introduction}
Appearing first in the 1960s, the concept of a Poincar\'e duality space proved to be a crucial tool in the surgery classification of high dimensional manifolds. It was Bill Browder who first introduced the notion of a space satisfying Poincar\'{e} duality with integer coefficients \cite{Browder0}, \cite{Browder1}. Browder's  definition is well-suited for doing simply connected surgery but it is inadequate  in the general case. To take into account some of the subtleties of the fundamental group, Spivak later introduced another definition \cite{Spivak}. Although Spivak's definition is adequate in the absolute setting, it is not sufficient for doing surgery in the case of Poincar\'{e} pairs.\footnote{This can be seen from our Theorem \ref{thm_interior_duality} below and Section \ref{sec_history}.} Slightly thereafter, Wall \cite{Wall1} introduced a definition of Poincar\'e  pair which controls both the interior and the boundary sufficiently. We remark that all of the above definitions are equivalent in the simply connected case.

In all the definitions of Poincar\'e space, one sacrifices the local homogeneity property that manifolds possess  and one instead focuses on the overall global structure. This has both advantages and disadvantages. A principal advantage is that many of the standard tools of algebraic topology become available. A disadvantage is that one sometimes desires to have information about local structure.

The current work had its genesis in a program of the authors on manifolds. In our investigations we needed a relative version of the {\it fibration theorem}
that says, roughly, in the presence of suitable finiteness assumptions, the total space of a fibration satisfies Poincar\'e duality if and only if its base and fiber do.\footnote{The result was announced without proof by Quinn in 1972 \cite[rem.~1.6]{Quinn}, and proved by Gottlieb \cite{Gottlieb} in 1979 using manifold techniques. A homotopy theoretic proof was given
by the first author in \cite{Klein3}.} The fibration theorem was only stated in the absolute case, and the relative case remained an open problem. We originally thought that the relative case would follow {\it mutatis mutandis} by  performing a reduction to the absolute case. However, such a reduction turned out to be more difficult and subtle than the authors had originally anticipated. The solution to this in the relative case appears as Theorem \ref{thm_fibration} below.

In proving the relative fibration theorem, we had to apply some foundational results on Poincar\'{e} pairs. However, after carefully checking the literature, we found that some foundational works were not clear and some old problems still remained unsolved. For example, to the best of our knowledge,  Wall's conjecture concerning the relationship between Poincar\'e pairs and their boundaries had not been solved (see Corollary \ref{cor_Wall} below). The ambiguity in these works partially arose because the definition of Poincar\'{e} pairs is not historically uniform (actually, the non-uniformity had been indicated by \cite[sec.~2]{Klein3}). Specifically, Wall's definition \cite[p.\ 215]{Wall1} is different from Spivak's \cite[p.\ 82]{Spivak}. In order to deal with this matter, we had to formulate and prove Theorem \ref{thm_P_crt} below though related (but different) results had been in the literature. A goal of this paper is to resolve the ambiguity and prove some new results.

Wall's definition of Poincar\'{e} pair has been mostly used in surgery theory. Our definition follows Wall's.

In what follows we say $(X,Y)$ is a space pair if $Y \subseteq X$ is a subspace, where $X$ is usually nonempty and $Y$ may be empty. When $Y$ is empty, the statements for $Y$ are considered vacuously true. We say $(X,Y)$ is a (path)-connected pair if $X$ is a (path)-connected space. By an \textit{orientation system} on $X$, we mean an local system on $X$ of infinite cyclic groups.

\begin{definition}\label{def_Poincare_pair}
A pair of topological spaces $(X, Y)$ is a \textit{Poincar\'{e} pair} of \textit{formal dimension} $n$ with $n \geq 0$ if there exist an orientation system $\mathcal{O}$ and a singular homology class $[X] \in H_{n} (X, Y; \mathcal{O})$ such that
\begin{enumerate}[\hspace{0.6cm}(a)]
\item the homomorphisms
\begin{equation}\label{def_Poincare_pair_1}
(\forall \mathcal{B}, \forall m) \ \ [X] \cap\:  H^{m} (X; \mathcal{B}) \overset{\cong}{\longrightarrow} H_{n-m} (X, Y; \mathcal{O} \otimes \mathcal{B})
\end{equation}
and
\begin{equation}\label{def_Poincare_pair_2}
(\forall \mathcal{B}, \forall m) \ \ [X] \cap\: H^{m} (X, Y; \mathcal{B}) \overset{\cong}{\longrightarrow} H_{n-m} (X; \mathcal{O} \otimes \mathcal{B})
\end{equation}
given by cap product with $[X]$ are isomorphisms, where $\mathcal{B}$ is a local system on $X$;

\item the homomorphisms
\begin{equation}\label{def_Poincare_pair_3}
(\forall \mathcal{G}, \forall m) \ \ \partial_{*} [X] \cap \: H^{m} (Y; \mathcal{G}) \overset{\cong}{\longrightarrow} H_{n-1-m} (Y; \mathcal{O}|_{Y} \otimes \mathcal{G})
\end{equation}
are isomorphisms, where $\mathcal{G}$ is a local system on $Y$ and $\partial_{*} [X]$ is the image of $[X]$ under the boundary homomorphism $\partial_{*}\: H_{n} (X, Y; \mathcal{O}) \rightarrow H_{n-1} (Y; \mathcal{O}|_{Y})$.
\end{enumerate}

Here, we are taking singular (co-)homology with twisted coefficients. If the above conditions hold, we call $\mathcal{O}$ the \textit{dualizing system} of $(X,Y)$ and $[X]$ a \textit{fundamental class}.
\end{definition}

\begin{remark}
It is not difficult to prove that for a Poincar\'{e} pair $(X,Y)$, the pair $(\mathcal O, n)$ is determined up to unique isomorphism. The fundamental class is not unique: for a given local system $\mathcal O$, if $(\mathcal O,[X])$ equips $(X,Y)$ with the structure of a Poincar\'e pair, then so does $(\mathcal O,-[X])$. If $X$ is path-connected, these are the only possibilities. However, the following uniqueness statement does hold for a Poincar\'{e} pair $(X,Y)$: any  pair $(\mathcal O, [X])$ satisfying the definition is defined up to unique isomorphism: if $(\mathcal O',[X]')$ is another such pair for $(X,Y)$, then there exists precisely one isomorphism $\phi\: \mathcal O\to \mathcal O'$ such that such that $\phi_\ast([X]) = [X]'$.
\end{remark}

\begin{remark}
If the space $Y$ in Definition \ref{def_Poincare_pair} is empty, then the above (\ref{def_Poincare_pair_3}) is vacuously true, and the isomorphisms \eqref{def_Poincare_pair_1} and \eqref{def_Poincare_pair_2} reduce to the single isomorphism
\[
[X] \cap\: H^{m} (X; \mathcal{B}) \overset{\cong}{\longrightarrow} H_{n-m} (X; \mathcal{O} \otimes \mathcal{B})\, . \ \
\]
In this instance we say that $X$ is a \textit{Poincar\'{e} space}.
\end{remark}

\begin{remark}
In Definition \ref{def_Poincare_pair}, we do not assume that $X$ is connected. Applying \eqref{def_Poincare_pair_2} with $\mathcal B = \mathbb Z$, we obtain an isomorphism
\[
[X] \cap\:  H^{n} (X, Y; \mathcal{O}) \overset{\cong}{\longrightarrow} H_{0} (X; \mathbb{Z}).
\]
Since $[X]$ is finitely supported,  the last isomorphism implies that $X$ has only finitely many path components.  We typically require that $X$ and $Y$ have the homotopy type of CW complexes.
\end{remark}

\begin{remark}
In the finitely dominated case, it is enough to check the duality isomorphisms for a single local system $\mathcal B$ defined by the fundamental group. Precisely, when $X$ is path connected we denote by $\mathbf{\Lambda}$ the local system on $X$ associated with $\mathbb{Z} [\pi_{1} (X)]$. If $(X,Y)$ is a finitely dominated pair (cf.~Definition \ref{def_finite_dominated} and Remark \ref{rmk_pair}),  Wall (see \cite[lem.~1.2]{Wall1} and Lemma \ref{lem_rd_dual} in this paper) proved that \eqref{def_Poincare_pair_1} is equivalent to the isomorphisms:
\[
(\forall m) \ \ [X] \cap\: H^{m} (X; \mathbf{\Lambda}) \overset{\cong}{\longrightarrow} H_{n-m} (X, Y; \mathcal{O} \otimes \mathbf{\Lambda}) \, ,\ \
\]
and \eqref{def_Poincare_pair_2} is equivalent to the isomorphisms:
\[
(\forall m) \ \ [X] \cap\: H^{m} (X, Y; \mathbf{\Lambda}) \overset{\cong}{\longrightarrow} H_{n-m} (X; \mathcal{O} \otimes \mathbf{\Lambda})\, . \ \
\]
\end{remark}

One family of our main results are on the relations among \eqref{def_Poincare_pair_1}, \eqref{def_Poincare_pair_2} and (\ref{def_Poincare_pair_3}). As mentioned in \cite[p.\ 797]{Goodwillie_Klein}, there are some redundancies in Definition \ref{def_Poincare_pair}. It's trivial to prove ``$\eqref{def_Poincare_pair_1} + (\ref{def_Poincare_pair_3}) \Rightarrow (\ref{def_Poincare_pair_2})$" and ``$\eqref{def_Poincare_pair_2}+ (\ref{def_Poincare_pair_3}) \Rightarrow \eqref{def_Poincare_pair_1}$". However, we shall construct the following example which shows that ``$\eqref{def_Poincare_pair_1} + (\ref{def_Poincare_pair_2})$" does not imply (\ref{def_Poincare_pair_3}).

\begin{bigthm}[$(\ref{def_Poincare_pair_1}) + (\ref{def_Poincare_pair_2}) \nRightarrow (\ref{def_Poincare_pair_3})$]\label{thm_interior_duality}
For $n \ge 1$, let $A$ be the punctured Poincar\'e homology $3$-sphere and let $C(A\times S^{n-1})$ denote the cone on $A\times S^{n-1}$. Then $(X,Y) := (C(A\times S^{n-1}),A\times S^{n-1})$ is a finite CW pair which satisfies \eqref{def_Poincare_pair_1} and \eqref{def_Poincare_pair_2}, but $Y$ is not a Poincar\'{e} space of any formal dimension.
\end{bigthm}

In the finitely dominated case we shall also prove the following equivalence.

\begin{bigthm}[Equivalence]\label{thm_P_eqv}
Let $(X, Y)$ be a  pair of finitely dominated spaces. Suppose we are given $(\mathcal{O},[X])$ in which $\mathcal{O}$ is an orientation system on $X$ and $[X] \in H_{n} (X, Y; \mathcal{O})$ is a homology class. Then with respect to these data \eqref{def_Poincare_pair_1} is equivalent to \eqref{def_Poincare_pair_2}.
\end{bigthm}

As an application, we will see that Theorem \ref{thm_P_eqv} implies the following conjecture of C.T.C.~Wall:

\begin{bigcor}[Wall's Conjecture]\label{cor_Wall}
With respect to the assumptions of Theorem \ref{thm_P_eqv}, assume further \eqref{def_Poincare_pair_1} \textit{or} (\ref{def_Poincare_pair_2}). Suppose $Y$ is a Poincar\'{e} space of formal dimension $n-1$. Then $(X, Y)$ is a Poincar\'{e} pair with fundamental class $[X]$.
\end{bigcor}

\begin{remark}
Here is an explanation of the context of Corollary \ref{cor_Wall}. In \cite{Wall1}, Wall defines Poincar\'{e} pairs as space pairs satisfying $``(\ref{def_Poincare_pair_1}) + (\ref{def_Poincare_pair_3})"$. In the second paragraph appearing of page 216, Wall writes: \textit{``We conjecture that the two requirements enclosed in square brackets are redundant; however, we will have to impose them."}
Here, Wall is referring to the requirements (i) the restriction $\mathcal{O}|_{Y}$ is the dualizing system of $Y$, and (ii) the class $\partial_{*} [X]$ is a fundamental class for $Y$. Note too that Wall is working in the finitely dominated context.
\end{remark}

Theorem \ref{thm_P_eqv} has an additional application.

\begin{bigcor}\label{cor_pair}
Under the assumptions of Theorem \ref{thm_P_eqv}, assume further \eqref{def_Poincare_pair_1} \textit{or} (\ref{def_Poincare_pair_2}). Suppose, for each component $Z$ of $Y$, every local system on $Z$ is the restriction of one on $X$. Then $(X, Y)$ is a Poincar\'{e} pair with  fundamental class $[X]$.
\end{bigcor}

Our other results rely heavily on the following Theorem \ref{thm_P_crt} which describes Poincar\'{e} pairs homotopy theoretically. We recall the concept of a Poincar\'{e} triad (compare with \cite[p.\ 216]{Wall1} and \cite[def.~C.5]{Goodwillie_Klein}).

We say $(X; Y_{1}, Y_{2})$ is a space triad (or 3-ad) if $Y_{1}$ and $Y_{2}$ are subspaces of $X$. We say it is (path)-connected if $X$ is.

\begin{definition}\label{def_Poincare_triad}
We say a space triad $(X; Y_{1}, Y_{2})$ is a \textit{Poincar\'{e} triad} of formal dimension $n$ if it satisfies the following conditions:
\begin{itemize}
\item $(X, Y_{1} \cup Y_{2})$ is a Poincar\'{e} pair of formal dimension $n$;
\item $\{ Y_{1}, Y_{2} \}$ is an excisive couple (cf. Definition \ref{def_excisive_couple});
\item  If $Y_{1}$ (resp.~$Y_{2}$) is nonempty, then $(Y_{1}, Y_{1} \cap Y_{2})$ (resp.~$(Y_{2}, Y_{1} \cap Y_{2})$) is a Poincar\'{e} pair of formal dimension $n-1$, and its fundamental classes are the restrictions (by excision) of those of $Y_{1} \cup Y_{2}$.
\end{itemize}
\end{definition}

In the above, $Y_{1}$ or $Y_{2}$ may be empty. If  both are empty, one obtains a Poincar\'{e} space. If one of them is empty, one obtains a Poincar\'{e} pair.

In the following Theorem \ref{thm_P_crt}, a Poincar\'e CW triad is a CW triad which is also a Poincar\'e triad.
Recall that a pair $(X,Y)$ is {\it $k$-connected} if for $j \le k$, every map $f\: (D^j,S^{j-1}) \to (X,Y)$ is homotopic rel $S^{j-1}$ to a map of $D^j$ into $Y$
(cf.~\cite[p.~70]{G.Whitehead}). When writing $F$ (or $F')$ we mean the homotopy fiber taken at any choice of basepoint, i.e., the fibers over all points of $X$ (or $Y_{2}$) should have the homotopy type $(k-1)$-spheres. If $Y_{2} = \emptyset$, the assertions for $F'$ will be considered vacuously true. When $Y_2$ is non-empty then we assume that the basepoints used in defining $F$ and $F'$ coincide. Note also that $Y_{2}$ may be disconnected.

\begin{bigthm}\label{thm_P_crt}
Suppose $(X; Y_{1}, Y_{2})$ is a Poincar\'{e} CW triad of formal dimension $n+k$ such that $X$ is connected and $Y_{1} \neq \emptyset$. Set $Y_0 := Y_{1} \cap Y_{2}$. Suppose $(X, Y_{1})$ and $(Y_{2}, Y_{0})$ are $2$-connected. Denote by $F$ (resp.~$F'$) a homotopy fiber of $Y_{1} \to X$ (resp.~$Y_{0} \to Y_{2}$). Then the following statements hold.
\begin{enumerate}
\item If $F$ is a homotopy $(k-1)$-sphere, then both \eqref{def_Poincare_pair_1} and \eqref{def_Poincare_pair_2} hold for the pair $(X,Y_2)$
 for some orientation system $\mathcal{O}$. Conversely, if  \eqref{def_Poincare_pair_1} \textit{or} \eqref{def_Poincare_pair_2} holds, then $F$ is a homotopy $(k-1)$-sphere.

\item $(X, Y_{2})$ is a Poincar\'{e} pair of formal dimension $n$ if and only if both $F$ and $F'$ are homotopy $(k-1)$-spheres.

\item If $(X, Y_{2})$ is a Poincar\'{e} pair of formal dimension $n$ and $Y_{2} \neq \emptyset$, then the natural inclusion $F' \to F$ is a homotopy equivalence.
\end{enumerate}
\end{bigthm}

\begin{remark}
Spivak formulated a statement \cite[prop.~4.6]{Spivak} similar to Part (2) of Theorem \ref{thm_P_crt} (cf.~\cite[thm.~I.4.2]{Browder2}). However, there is a major difference between Spivak's result and ours in that {\it Spivak ignored the homotopy fiber $F'$}. That is, Spivak asserted that $(X, Y_{2})$ is a Poincar\'{e} pair if and only if $F$ itself is a homotopy sphere. The distinction originates from the fact that Spivak's notion of Poincar\'{e} pair is different from ours (and Wall's). The issue will be explained in Section \ref{sec_history}.

We wish to emphasize here that by Theorem \ref{thm_interior_duality}, one has to take $F'$ into account if one uses our definition. Otherwise, the conclusion of Theorem \ref{thm_P_crt} will be no longer true (see Remark \ref{rmk_P_crt}).
\end{remark}

Theorem \ref{thm_P_crt} has several applications which we now develop.  It is well-known that if $(X_{1}, Y)$ and $(X_{2}, Y)$ are Poincar\'{e} pairs of formal dimension $n$, then the amalgamation $X_{1} \cup_{Y} X_{2}$ is a Poincar\'{e} space of formal dimension $n$ (see \cite[thm.~2.1, thm.~2.1~add.]{Wall1}, \cite[prop.~C.6]{Goodwillie_Klein} and Lemma \ref{lem_gluing} in this paper). One may ask about  the extent to which the following partial converse is true:  suppose that $X_{1} \cup_{Y} X_{2}$ is a Poincar\'{e} space of formal dimension $n$ and $Y$ is a Poincar\'{e} space of formal dimension $n-1$. Then are $(X_{1}, Y)$ and $(X_{2}, Y)$ Poincar\'{e} pairs of formal dimension $n$? In general the answer is no. Here is an easy counterexample: let $X_{1} = S^{1}$, $X_{2} = [0,1]$ and
\[
X_{1} \cup_{Y} X_{2} = X_{1} \vee X_{2},
\]
where $Y$ is the wedge point.

However, a special case of the following Theorem \ref{thm_doubling} implies that, roughly speaking, the answer is affirmative if the amalgamation is a double.

\begin{bigthm}[Doubling]\label{thm_doubling}
Suppose $(X; Y_{1}, Y_{2})$ is a finitely dominated triad (cf. Definition \ref{def_finite_dominated}) with $Y_{2} \neq \emptyset$ and $Y_{0} = Y_{1} \cap Y_{2}$. Suppose further $(Y_{2}, Y_{0})$ is a Poincar\'{e} pair of formal dimension $n-1$. Then $(X \cup_{Y_{2}} X, Y_{1} \cup_{Y_{0}} Y_{1})$ is a Poincar\'{e} pair of formal dimension $n$ if and only if $(X; Y_{1}, Y_{2})$ is Poincar\'{e} triad of formal dimension $n$.
\end{bigthm}

The merit of Theorem \ref{thm_doubling} lies in the reduction of a Poincar\'{e} triad (resp.~pair) to a Poincar\'{e} pair (resp.~space).
It will be used to prove the following Theorem \ref{thm_fibration} which was the original motivation for this paper.

\begin{bigthm}[Fibration]\label{thm_fibration}
Suppose $(B, \partial B)$ and $(F, \partial F)$ are finitely dominated pairs such that $B$ is connected and based. Let
\[
(F, \partial F) \rightarrow (E, \partial_{1} E) \rightarrow B
\]
be a pair of fibrations over $B$, where $\partial_{1} E \rightarrow E$ is a closed cofibration.  Let $\partial_{2} E = E|_{\partial B}$ denote the pullback of $E \rightarrow B$ induced by $\partial B \subset B$. Assume further $(E; \partial_{1} E, \partial_{2} E)$ is a finitely dominated triad.

Then the following statements hold:

\begin{enumerate}
\item The triad $(E; \partial_{1} E, \partial_{2} E)$ is a Poincar\'{e} triad if and only if
$(B, \partial B)$ and $(F, \partial F)$ are Poincar\'{e} pairs.

\item If $(E; \partial_{1} E, \partial_{2} E)$ is a Poincar\'{e} triad, then its formal dimension is the sum of those of $(B, \partial B)$ and $(F, \partial F)$;
\end{enumerate}
\end{bigthm}

A special case of Theorem \ref{thm_fibration} is when $\partial B = \emptyset$ and $\partial F = \emptyset$. We call this case the absolute version of the theorem. The absolute version was announced by Quinn and proved in \cite{Gottlieb} and \cite[cor.~F]{Klein2} by different methods. Another special case of Theorem \ref{thm_fibration} is the product of space pairs. Thus it also generalizes  Theorem 2.5 in \cite{Wall1}. In \cite{Wall1}, Wall proved that if $(B, \partial B)$ and $(F, \partial F)$ are finitely dominated Poincar\'{e} pairs, then so is the pair $(B \times F, B \times \partial F \cup \partial B \times F)$. However, Wall neither stated nor proved the converse. The proof of Theorem \ref{thm_fibration} will be obtained by a reduction to the absolute version by virtue of Theorem \ref{thm_doubling}.

Lastly, we also prove the following result on finite coverings.
\begin{bigthm}[Finite Covering]\label{thm_covering}
Suppose $(\bar{X}, \bar{Y})$ is a finite covering space of $(X, Y)$. Then the following statements hold:

\begin{enumerate}
\item If $(X, Y)$ is a Poincar\'{e} pair of formal dimension $n$, then so is $(\bar{X}, \bar{Y})$ and a fundamental class for $(\bar{X}, \bar{Y})$ is given by the transfer of a fundamental class for $(X, Y)$.

\item Assume that $(X, Y)$ is a finitely dominated pair. If $(\bar{X}, \bar{Y})$ is a Poincar\'{e} pair of formal dimension $n$, then so is $(X, Y)$.
\end{enumerate}
\end{bigthm}

Some remarks about Part (2) of Theorem \ref{thm_covering} are in order. Clearly, the result is a special case of Theorem \ref{thm_fibration}. However, our proof of Theorem \ref{thm_fibration} uses Theorem \ref{thm_covering}. When $Y = \emptyset$, Part (2) of Theorem \ref{thm_covering} is well-known. However, the relative version appears to be new. In fact, our proof requires Theorem \ref{thm_P_crt}.

Note that the finite domination assumption in Part (2) is needed: the double cover $S^\infty$ of $\Bbb RP^\infty$ is contractible. Consequently, $S^\infty$ is a Poincar\'e space of dimension 0, but $\Bbb RP^\infty$ is not a Poincar\'e space.

\begin{out}
Section \ref{sec_preliminary} recalls some facts frequently used in the paper. In Section \ref{sec_Thom}, we prove and recall some foundational results on the Thom isomorphism. Theorems \ref{thm_interior_duality}, \ref{thm_P_eqv} and \ref{thm_P_crt}, and Corollaries \ref{cor_Wall} and \ref{cor_pair} are proved in Section \ref{sec_dual}. Theorems \ref{thm_doubling}, \ref{thm_covering} and \ref{thm_fibration} are proved in Sections \ref{sec_double}, \ref{sec_cover} and \ref{sec_fibration} respectively. Lastly, in Section \ref{sec_history}, we discuss the various definitions of Poincar\'{e} pairs.
\end{out}

\begin{acks} The authors are indebted to Bill Browder for explaining to us the genesis of Poincar\'e duality spaces. The first author is especially grateful to Bill for his inspiration and support. We thank an anonymous referee who meticulously read the paper and made many helpful suggestions which led to an improved presentation. The first author was partially supported by Simons Foundation Collaboration Grant 317496. The second author was partially supported by NSFC11871272. The third author was partially supported by NSFC11571343.
\end{acks}

\section{Preliminaries}\label{sec_preliminary}
The purpose of this section is to recall some basic results that are
frequently used in the paper. In the sequel, $R$ denotes an associative ring with unit.

\begin{lemma}\label{lem_acyclic}
Assume $C_{\bullet}$ is a chain complex of right (resp. left) projective $R$-modules. Assume further that $C_{\bullet}$ is connective, i.e., $C_{p} =0$ for all $p<0$. Then following statements are equivalent:
\begin{enumerate}
\item $C_{\bullet}$ is contractible.

\item $C_{\bullet}$ is acyclic, i.e., the homology groups $H_{p} (C_{\bullet})$ vanish for all $p$.

\item The homology groups $H_{p} (C_{\bullet}; M) = 0$ vanish for all $p$ and all left (resp. right) $R$-modules $M$.

\item The cohomology groups $H^{p} (C_{\bullet}; M)$ vanish for all $p$ and all right (resp. left) $R$-modules $M$.
\end{enumerate}
\end{lemma}
\begin{proof}
We only prove the case that $C_{\bullet}$ is a chain complex of right $R$-modules. The proof in the case of a left $R$-complex is similar.

Clearly, (1) implies all of the others. Trivially, $(3) \Rightarrow (2)$. Since $C_{\bullet}$ is degree-wise projective, it is well-known that $(2) \Rightarrow (1)$. To finish the proof, it suffices to prove $(4) \Rightarrow (2)$.

Assuming (4), we prove (2) using the Universal Coefficient Spectral Sequence (see e.g., \cite[thm.~(2.3)]{Levine}), whose $E_{2}$ terms are
\[
E_{2}^{p,q} = \mathrm{Ext}^{q}_{R} (H_{p} (C_{\bullet}), M).
\]
The spectral sequence converges to $H^{*} (C_{\bullet}; M)$. Clearly, $E_{2}^{p,q} = 0$ for all $q<0$. Since $C_{p} =0$ for all $p<0$, we see that $E_{2}^{p,q} = 0$ for all $p<0$. Then
\[
\mathrm{Hom}_{R} (H_{0}(C_{\bullet}), M) = E_{2}^{0,0} = E_{\infty}^{0,0} = 0
\]
for all $M$. Taking $M = H_{0}(C_{\bullet})$, we infer $H_{0}(C_{\bullet}) = 0$. Then we know $E_{2}^{0,q} = 0$ for all $q$ and $M$. We further infer
\[
\mathrm{Hom}_{R} (H_{1}(C_{\bullet}), M) = E_{2}^{1,0} = E_{\infty}^{1,0} = 0,
\]
which implies $H_{1}(C_{\bullet}) = 0$. By induction on $p$, we see that $H_{p}(C_{\bullet}) = 0$ for all $p$, which completes the proof.
\end{proof}

\begin{corollary}\label{cor_complex_isomorphism}
For a homomorphism of $f\: C_{\bullet} \to D_{\bullet}$
 of connective chain complexes of projective right (resp.~left)  $R$-modules,
the  following statements are equivalent:
\begin{enumerate}
\item $f$ is a chain homotopy equivalence.

\item $f$ is a quasi-isomorphism, i.e., the homomorphism $f_{*}:  H_{p} (C_{\bullet}) \rightarrow H_{p} (D_{\bullet})$ is an isomorphism for all $p$.

\item The homomorphism $f_{*}:  H_{p} (C_{\bullet}; M) \rightarrow H_{p} (D_{\bullet}; M)$ is an isomorphism for all $p$ and all left (resp. right) $R$-modules $M$.

\item The homomorphism $f^{*}:  H^{p} (D_{\bullet}; M) \rightarrow H^{p} (C_{\bullet}; M)$ is an isomorphism for all $p$ and all right (resp. left) $R$-modules $M$.
\end{enumerate}
\end{corollary}
\begin{proof}
Apply Lemma \ref{lem_acyclic} to the mapping cone.
\end{proof}

\subsection{Local systems}
For a space $X$ we let $\pi_X$ be its fundamental groupoid. This is the category whose objects are the points $x$ of $X$ where a morphism $x \rightarrow y$ is a path homotopy class starting at $x$ and terminating at $y$. A {\it local system} of abelian groups on $X$ is a functor
\[
\mathcal{B} \colon\! \pi_{X} \rightarrow \mathrm{Ab} \, ,
\]
where $\mathrm{Ab}$ denotes the category of abelian groups. If $X$ is path-connected, then a local system $\mathcal B$ is determined up to canonical isomorphism by its value $\mathcal{B}_{x_{0}}$ at a fixed base point $x_{0} \in X$ together with an action of $\pi_{1} (X, x_{0})$ on $\mathcal{B}_{x_{0}}$.

In particular, for a pair of spaces $(X,Y)$, one has the associated twisted singular chain and cochain complexes
\[
C_{\bullet} (X,Y; \mathcal{B}) \quad \text{ and } \quad C^{\bullet} (X,Y; \mathcal{B})\, .
\]
We recall how these are defined. In what follows, we may assume $X$ is path-connected by working over each path component. Let $\tilde X \rightarrow X$ be the universal cover associated with a fixed choice of basepoint $x_{0} \in X$. Denote by $\mathbf{\Lambda}$ the local system on $X$ associated with the group ring $\Lambda := \mathbb{Z}[\pi_{1} (X, x_{0})]$. Note that $\Lambda$ is a bimodule over itself. Denote by $\tilde Y = Y \times_{X} \tilde X$, i.e., the pullback of $\tilde X\to X$ along $Y \to X$. Then
\[
C_{\bullet} (X,Y; \mathbf \Lambda) := C_{\bullet} (\tilde X, \tilde Y; \mathbb{Z})
\]
is a left $\Lambda$-complex, where the $\pi_{1} (X, x_{0})$-action is given by deck transformations. The involution $g \mapsto g^{-1}$ of $\pi_{1} (X, x_{0})$ induces one on $\Lambda$ by extending linearly. Denote this involution by $\lambda \mapsto \bar \lambda$. Then the involution equips $C_{\bullet} (X,Y; \mathbf{\Lambda})$ with the structure of a right $\Lambda$-complex in which $\xi \cdot \lambda := \bar{\lambda} \cdot \xi$. Hence, $C_{\bullet} (X,Y; \mathbf{\Lambda})$ is both a left and a right $\Lambda$-complex.

Suppose $\mathcal{B}$ is a local system on $X$. Let
\[
M = \mathcal{B}_{x_{0}}.
\]
Then $M$ is a right $\Lambda$-module, and we have
\begin{equation}\label{equ_singular_complex}
C_{\bullet} (X,Y; \mathcal{B}) := M \otimes_{\Lambda} C_{\bullet} (X,Y; \mathbf{\Lambda}) \ \text{and} \ C^{\bullet} (X,Y; \mathcal{B}) := \mathrm{Hom}_{\Lambda} (C_{\bullet} (X,Y; \mathbf{\Lambda}), M).
\end{equation}

Note that $C_{\bullet} (X,Y; \mathbf{\Lambda})$ is a free $\Lambda$-complex. Therefore, we can apply Corollary \ref{cor_complex_isomorphism} (taking $R = \Lambda$) to $C_{\bullet} (X,Y; \mathcal{B})$ and $C^{\bullet} (X,Y; \mathcal{B})$.

\begin{remark}
Let $(X,Y)$ be a space pair. Then for local systems $\mathcal G,\mathcal H$, one has cup products
\[
C^{p} (X,Y; \mathcal{G}) \times C^{q} (X; \mathcal{H}) \overset{\cup}{\longrightarrow} C^{p+q} (X,Y; \mathcal{G} \otimes \mathcal{H})\, ,
\]
as well as cap products
\[
C_{p} (X,Y; \mathcal{G}) \times C^{q} (X; \mathcal{H}) \overset{\cap}{\longrightarrow} C_{p-q} (X,Y; \mathcal{G} \otimes \mathcal{H})\, ,
\]
and
\[
C_{p} (X,Y; \mathcal{G}) \times C^{q} (X,Y; \mathcal{H}) \overset{\cap}{\longrightarrow} C_{p-q} (X; \mathcal{G} \otimes \mathcal{H})\, .
\]
Here, $\mathcal{G} \otimes \mathcal{H}$ is the local system defined by the composition of functors
\[
\begin{CD}
\pi_X @> (\mathcal{G} ,\mathcal{H}) >> \mathrm{Ab} \times \mathrm{Ab} @> \otimes >> \mathrm{Ab}\, ,
\end{CD}
\]
where in the display, $\otimes$ denotes the tensor product operation on abelian groups.
\end{remark}

\begin{remark}\label{rmk_simplicial_complex}
Suppose $(X,Y)$ is a simplicial pair. We use $\Delta_{\bullet}$ and $\Delta^{\bullet}$ to denote the corresponding simplicial chain and cochain complexes. Clearly, $\Delta_{\bullet} (X,Y; \mathbf{\Lambda})$ is $\Lambda$-free,
\[
\Delta_{\bullet} (X,Y; \mathcal{B}) = M \otimes_{\Lambda} \Delta_{\bullet} (X,Y; \mathbf{\Lambda})  \ \ \text{and} \ \ \Delta^{\bullet} (X,Y; \mathcal{B}) = \mathrm{Hom}_{\Lambda} (\Delta_{\bullet} (X,Y; \mathbf{\Lambda}), M).
\]
Each of these is chain equivalent to its corresponding total singular chain complex, and the above results hold for these complexes as well. If we further assume $(X,Y)$ is finite, then $\Delta^{\bullet} (X,Y; \mathbf{\Lambda})$ is also $\Lambda$-free and
\[
\Delta^{\bullet} (X,Y; \mathcal{B}) = M \otimes_{\Lambda} \Delta^{\bullet} (X,Y; \mathbf{\Lambda}).
\]
\end{remark}


\subsection{Reduction to the finite case}\label{sec_finite}
We introduce the following Lemmas \ref{lem_rd_space} and \ref{lem_rd_dual} which allow us to reduce a finitely dominated pair (triad) to a finite CW pair (triad).

A triple $(A; B_{1}, B_{2})$ is said to be a {\it space triad}  if $A$ is a topological space and $B_{i}$ are subspaces of $A$, where $B_{i}$ may be empty \cite[cf.~p.~8]{G.Whitehead}. When $A$ is a (finite) CW complex and $B_{i}$ are subcomplexes of $A$, we say $(A; B_{1}, B_{2})$ is a (finite) CW triad.
A map of space triads
\[
f: (A; B_{1}, B_{2}) \rightarrow (C; D_{1}, D_{2}),
\]
is a map of spaces $A\to C$ such that $f(B_i) \subset D_i$ for $i=1,2$.
If $I=[0,1]$ is the unit interval, then we may form the triad $(I\times A; I\times B_{1}, I\times B_{2})$.
If
\[
f_0,f_1\: (A; B_{1}, B_{2}) \rightarrow (C; D_{1}, D_{2})
\]
are maps of triads, then a homotopy from $f_0$ to $f_1$ is a map of triads
\[
H: (I \times A; I \times B_{1}, I \times B_{2}) \rightarrow (C; D_{1}, D_{2})
\]
such that $H_{0} = f_{0}$ and $H_{1} = f_{1}$, where $H_i$ denotes the restriction of $H$ to $i\times A$ for $i = 0,1$. We write $f_0 \simeq f_1$ in this case.

\begin{definition}\label{def_homotopy_triad}
A map of triads $f\: (A; B_{1}, B_{2})\rightarrow (C; D_{1}, D_{2})$ is called a \textit{homotopy equivalence of triads}
if there exists a map of triads $g\: (C; D_{1}, D_{2}) \rightarrow (A; B_{1}, B_{2})$ such that $f\circ g \simeq \mathrm{id}$ and $g\circ f \simeq \mathrm{id}$.
\end{definition}

\begin{definition}\label{def_finite_dominated}
Suppose $(A; B_{1}, B_{2})$ is a space triad and $(K; L_{1}, L_{2})$ is a finite CW triad. If there exist maps of triads
\[
(A; B_{1}, B_{2}) \rightarrow (K; L_{1}, L_{2}) \rightarrow (A; B_{1}, B_{2})
\]
such that the composition is homotopic to the identity, then $(A; B_{1}, B_{2})$ is said to be a \textit{finitely dominated triad}.
\end{definition}

\begin{remark}\label{rmk_pair}
In particular, if the above $B_{2}$ and $D_{2}$ are empty, then those space triads are reduced to space pairs $(A; B_{1})$ and $(C; D_{1})$. We also denote them by $(A, B_{1})$ and $(C, D_{1})$. Definitions \ref{def_homotopy_triad} and \ref{def_finite_dominated} are reduced to concepts associated with space pairs.
\end{remark}

\begin{remark}
There  exist finitely dominated Poincar\'{e} spaces which are not homotopy equivalent to any finite CW complex (see \cite[thm 1.5]{Wall1}).
\end{remark}

\begin{remark}
By \cite[cor.~1 and 2]{Browder3}, the fundamental group of a Poincar\'{e} space is finitely generated, and a Poincar\'{e} CW complex is finitely dominated if and only if its fundamental group is finitely presented. (See also \cite[p.\ 135]{Brown}.) Similar results also hold for Poincar\'{e} pairs. There do exist Poincar\'{e} spaces whose fundamental groups are not finitely presented, see \cite[thm.~C]{Davis}.
\end{remark}

\begin{lemma}\label{lem_rd_space}
Suppose $(A; B_{1}, B_{2})$ is a finitely dominated triad and $W$ is a finite CW complex with trivial Euler characteristic. Then there exist a finite CW triad $(K; L_{1}, L_{2})$ and a homotopy equivalence of triads
$f\: (K; L_{1}, L_{2}) \rightarrow (A \times W; B_{1} \times W, B_{2} \times W)$.
\end{lemma}
\begin{proof}
Since $B_{1} \cap B_{2}$ is finitely dominated, $W$ is finite and $\chi (W) =0$, by Gersten's product formula \cite[thm.~0.1]{Gersten}, Wall's finiteness obstruction of $(B_{1} \cap B_{2}) \times W$ is $0$. Hence, there exist a finite CW complex $Z$ and a homotopy equivalence $j: Z \rightarrow (B_{1} \cap B_{2}) \times W$. Similarly, there are homotopy equivalences $g\: P \rightarrow A \times W$, $h_{1}\: Q_{1} \rightarrow B_{1} \times W$ and $h_{2}\: Q_{2} \rightarrow B_{2} \times W$, where $P$, $Q_{1}$ and $Q_{2}$ are finite CW complexes.

Let $\hat{h}_{i}$ be a homotopy inverse of $h_{i}$. The map $\hat{h}_{i} j$ is homotopic to a cellular map $\varphi_{i}: Z \rightarrow Q_{i}$. Let $L_{i}$ be the mapping cylinder of $\varphi_{i}$. Obviously, there is a map $h'_{i}: L_{i} \rightarrow B_{i} \times W$ such that $h'_{i}|_{Q_{i}} = h_{i}$ and $h'_{i}|_{Z} = j$. Here $Z$ is a subspace of both $L_{1}$ and $L_{2}$. Clearly, $h'_{i}$ are homotopy equivalences. Gluing $L_{1}$ and $L_{2}$ along $R$, we obtain a space $L_{1} \cup L_{2}$ and a map
\[
h': (L_{1} \cup L_{2}; L_{1}, L_{2}) \rightarrow ((B_{1} \cup B_{2}) \times W; B_{1} \times W, B_{2} \times W)
\]
such that $h'|_{L_{i}} = h'_{i}$ and $L_{1} \cap L_{2} = R$.

Let $\hat{g}$ be a homotopy inverse of $g$. The map $\hat{g} h'$ is homotopic to a cellular map $\phi\: L_{1} \cup L_{2} \rightarrow P$. Let $K$ be the mapping cylinder of $\phi$. There is a map
\[
f: (K; L_{1}, L_{2}) \rightarrow (A \times W; B_{1} \times W, B_{2} \times W)
\]
such that $f|_{L_{1} \cup L_{2}} = h'$ and $f|_{P} = g$. Clearly, $(K; L_{1}, L_{2})$ is a finite CW triad. It's also easy to see $f$ is a singular homotopy equivalence in the sense of \cite[p.~274]{Milnor59}, i.e. a weak homotopy equivalence \cite[p.~220, (3.1)]{G.Whitehead}. Furthermore, $(A \times W; B_{1} \times W, B_{2} \times W)$ is a finitely dominated triad. Now the conclusion follows from \cite[thm.~2~\&~lem.~1]{Milnor59}.
\end{proof}

Assume $X$ is path-connected. Denote by $\mathbf{\Lambda}$ the local system associated with $\mathbb{Z} [\pi_1(X)]$ on $X$. Let $\eta \in H_{n} (X,Y; \mathcal{O})$ be a homology class, where  $\mathcal{O}$ is an orientation system on $X$. Suppose $\theta \in H_{3} (S^{3}; \mathbb{Z})$ is a fundamental class of $S^{3}$. Let $p_{1}: X \times S^{3} \rightarrow X$ denote the projection. In the following, $p_{1}^{*} \mathcal{O}$ is the pullback of $\mathcal{O}$ via $p_{1}$ (cf.~\cite[p.~265]{G.Whitehead}).

\begin{lemma}\label{lem_rd_dual}
Suppose $(X,Y)$ is a finitely dominated pair. Then the following statements (1)-(4) are equivalent.
\begin{enumerate}
\item The following are isomorphisms for all $*$:
\[
\eta \cap \: H^{*} (X; \mathbf{\Lambda}) \overset{\cong}{\longrightarrow} H_{n-*} (X, Y; \mathcal{O} \otimes \mathbf{\Lambda}).
\]

\item The following are isomorphisms for all $*$ and all local systems $\mathcal{G}$ on $X$:
\[
\eta \cap \: H^{*} (X; \mathcal{G}) \overset{\cong}{\longrightarrow} H_{n-*} (X, Y; \mathcal{O} \otimes \mathcal{G}).
\]

\item The following are isomorphisms for all $*$:
\[
(\eta \times \theta) \cap \: H^{*} (X \times S^{3}; \mathbf{\Lambda}) \overset{\cong}{\longrightarrow} H_{n+3-*} (X \times S^{3}, Y \times S^{3}; p_{1}^{*} \mathcal{O} \otimes \mathbf{\Lambda}).
\]

\item The following are isomorphisms for all $*$ and all local systems $\mathcal{G}$ on $X \times S^{3}$:
\[
(\eta \times \theta) \cap \: H^{*} (X \times S^{3}; \mathcal{G}) \overset{\cong}{\longrightarrow} H_{n+3-*} (X \times S^{3}, Y \times S^{3}; p_{1}^{*} \mathcal{O} \otimes \mathcal{G}).
\]
\end{enumerate}

Furthermore, the following (5)-(8) are equivalent.
\begin{enumerate}
\setcounter{enumi}{4}

\item The following are isomorphisms for all $*$:
\[
\eta \cap \: H^{*} (X, Y; \mathbf{\Lambda}) \overset{\cong}{\longrightarrow} H_{n-*} (X; \mathcal{O} \otimes \mathbf{\Lambda}).
\]

\item The following are isomorphisms for all $*$ and all local systems $\mathcal{G}$ on $X$:
\[
\eta \cap \: H^{*} (X, Y; \mathcal{G}) \overset{\cong}{\longrightarrow} H_{n-*} (X; \mathcal{O} \otimes \mathcal{G}).
\]

\item The following are isomorphisms for all $*$:
\[
(\eta \times \theta) \cap \: H^{*} (X \times S^{3}, Y \times S^{3}; \mathbf{\Lambda}) \overset{\cong}{\longrightarrow} H_{n+3-*} (X \times S^{3}; p_{1}^{*} \mathcal{O} \otimes \mathbf{\Lambda}).
\]

\item The following are isomorphisms for all $*$ and all local systems $\mathcal{G}$ on $X \times S^{3}$:
\[
(\eta \times \theta) \cap \: H^{*} (X \times S^{3}, Y \times S^{3}; \mathcal{G}) \overset{\cong}{\longrightarrow} H_{n+3-*} (X \times S^{3}; p_{1}^{*} \mathcal{O} \otimes \mathcal{G}).
\]
\end{enumerate}
\end{lemma}

\begin{remark}
The equivalences $(1) \Leftrightarrow (2)$ and $(5) \Leftrightarrow (6)$ in Lemma \ref{lem_rd_dual} are not new, see e.g. \cite[lem.~1.2]{Wall1}. Nevertheless, for the convenience of readers, we shall give a slightly different proof.
\end{remark}

\begin{remark} In the above,
rather than taking product with $S^{3}$, we could have instead taken the product with $S^{1}$ at the expense of changing the fundamental group.
\end{remark}

\begin{proof}[Proof of Lemma \ref{lem_rd_dual}]
We shall only prove the equivalence from (1) to (4). The argument for (5)-(8) is similar.

Firstly, we prove the equivalences $(1) \Leftrightarrow (3)$ and $(2) \Leftrightarrow (4)$.

By Theorem \ref{thm_Kunneth_cross}, the following square commutes up to sign.
\[
\xymatrix{
  \underset{q+r=*}{\bigoplus} H^{q} (X; \mathcal{G}) \otimes H^{r} (S^{3}; \mathbb{Z}) \ar[d]_{(\eta \cap) \otimes (\theta \cap)} \ar[r]^-{\times}_-{\cong} & H^{*} (X \times S^{3}; p_{1}^{*} \mathcal{G}) \ar[d]^{(\eta \times \theta) \cap} \\
  \underset{q+r=*}{\bigoplus} H_{n-q} (X, Y; \mathcal{G}) \otimes H_{3-r} (S^{3}; \mathbb{Z}) \ar[r]^-{\times}_-{\cong} & H_{n+3-*} (X \times S^{3}, Y \times S^{3}; p_{1}^{*} \mathcal{O} \otimes p_{1}^{*} \mathcal{G})\, .    }
\]
Since $H_{\bullet} (S^{3}; \mathbb{Z})$ and $H^{\bullet} (S^{3}; \mathbb{Z})$ are free, by the K\"{u}nneth Theorems \ref{thm_Kunneth_h} and \ref{thm_Kunneth_c}, the two
horizontal maps in this square are isomorphisms. Thus the right vertical map
 is an isomorphism if and only if the left one is. Again by the freeness of $H_{\bullet} (S^{3}; \mathbb{Z})$ and $H^{\bullet} (S^{3}; \mathbb{Z})$, the left vertical isomorphism is equivalent to the isomorphism
\[
\eta \cap\: H^{q} (X; \mathcal{G}) \overset{\cong}{\longrightarrow} H_{n-q} (X, Y; \mathcal{G}).
\]
Since ${p_{1}}_{*}: \pi_{1} (X \times S^{3}) \rightarrow \pi_{1} (X)$ is an isomorphism, we infer $(1) \Leftrightarrow (3)$ and $(2) \Leftrightarrow (4)$.

Trivially, we have $(4) \Rightarrow (3)$.  To finish the proof, it suffices to prove $(3) \Rightarrow (4)$.

By Lemma \ref{lem_rd_space}, there exists a finite simplicial pair $(K,L)$ such that
\[
(K,L) \simeq (X \times S^{3}, Y \times S^{3}).
\]
Assume the isomorphisms
\[
\eta \cap\: H^{*} (K; \mathbf{\Lambda}) \overset{\cong}{\longrightarrow} H_{n-*} (K,L; \mathcal{O}_{1} \otimes \mathbf{\Lambda})
\]
for some orientation system $\mathcal{O}_{1}$ on $K$.  As mentioned in Remark \ref{rmk_simplicial_complex}, we may work with the simplicial chain complexes $\Delta_{\bullet}$ and cochain complexes $\Delta^{\bullet}$. Choose a cycle $\alpha \in
\Delta_{n} (K,L; \mathcal{O}_{1})$ representing $\eta$. Then the following chain map
\[
\alpha \cap \colon\!  \Delta^{*}(K; \mathbf{\Lambda}) \rightarrow \Delta_{n-*} (K,L; \mathcal{O}_{1} \otimes \mathbf{\Lambda})
\]
induces an isomorphism between homology groups. Since $(K,L)$ is finite, we know that $\Delta^{*} (K; \mathbf{\Lambda}) = 0$ and $\Delta_{n-*} (K,L; \mathcal{O}_{1} \otimes \mathbf{\Lambda}) = 0$ when $*$ is small or large enough. Thus, by Remark \ref{rmk_simplicial_complex} and Corollary \ref{cor_complex_isomorphism}, we obtain the isomorphisms
\[
\eta \cap\: H^{*} (K; \mathcal{B}) \overset{\cong}{\longrightarrow} H_{n-*} (K,L; \mathcal{O}_{1} \otimes \mathcal{B})
\]
for all integers $*$ and all local systems $\mathcal{B}$. This provides
the implication $(3) \Rightarrow (4)$.
\end{proof}

\begin{remark}
Instead of proving $(1) \Leftrightarrow (2)$ in Lemma \ref{lem_rd_dual} directly, we reduced the conclusion for the finitely dominated pair $(X,Y)$ to that for the finite simplicial pair $(K,L)$. This reduction will be frequently employed in the paper.
\end{remark}

\section{The Thom Isomorphism}\label{sec_Thom}

We  define  below a version of the Thom isomorphism for space pairs and investigate its properties.

\begin{definition}[Thom Isomorphism]\label{def_Thom}
Let $(X,Y)$ be a space pair. Suppose $\mathcal{O}$ is an orientation system on $X$, and $u \in H^{k} (X,Y; \mathcal{O})$ with $k \geq 0$. If
\[
u \cup\: H^{*} (X; \mathcal{B}) \overset{\cong}{\longrightarrow} H^{k+*} (X,Y; \mathcal{O} \otimes \mathcal{B})
\]
is an isomorphism for all integers $*$ and local systems $\mathcal{B}$ on $X$, then we say $(X,Y)$ satisfies the \textit{Thom isomorphism} of formal dimension $k$. We call $u$ a \textit{Thom class} of $(X,Y)$.
\end{definition}

\begin{remark}
Suppose $(X,Y)$ satisfies the Thom isomorphism of formal dimension $k$. We claim that $k>0$ if and only if $\pi_{0} (Y) \to \pi_{0} (X)$ is surjective. Assuming $\pi_0(Y) \to \pi_0(X)$ is surjective, we infer $H^{0} (X,Y; \mathcal{O}) = 0$. If $k=0$, we would have $H^{0} (X; \mathbb{Z}) =0$ which is impossible since $X$ is non-empty. Conversely, assuming $k>0$, we get $H^{0} (X,Y; \mathbb{Z}) \cong H^{-k} (X; \mathcal{O}) = 0$. Obviously, $\pi_{0} (Y) \to \pi_{0}(X)$ is surjective.
\end{remark}

\begin{remark}
It's easy to see that $(X,Y)$ satisfies the Thom isomorphism of formal dimension $0$ if and only if $Y$ is empty.
\end{remark}

\begin{theorem}[Thom Isomorphism: Uniqueness]\label{thm_T_unq}
Suppose $\mathcal{O}_{1}$ and $\mathcal{O}_{2}$ are two orientation systems on a path-connected space $X$. Suppose both $u_{1} \in H^{k_{1}} (X, Y; \mathcal{O}_{1})$ and $u_{2} \in H^{k_{2}} (X, Y; \mathcal{O}_{2})$ are Thom classes. Then $k_{1} = k_{2}$ and there is a unique isomorphism $\Phi\: \mathcal{O}_{1} \overset{\cong}{\rightarrow} \mathcal{O}_{2}$ covering identity map of $X$ such that $\Phi^{*} (u_{2}) = u_{1}$.
\end{theorem}
\begin{proof}
Since $u_{1}$ is a Thom class of formal dimension $k_{1}$, by Definition \ref{def_Thom}, we see that $H^{*} (X,Y; \mathcal{B}) =0$ for all local systems $\mathcal{B}$ and all $* < k_{1}$. Furthermore,
\[
H^{k_{1}} (X,Y; \mathcal{O}_{1}) \cong H^{0} (X; \mathbb{Z}) \cong \mathbb{Z} \neq 0.
\]
In other words, there exists a local system $\mathcal{B}$ such that $H^{k_{1}} (X,Y; \mathcal{B}) \neq 0$. Applying the same argument to $u_{2}$, we see that $k_{1} = k_{2}$.

Denote by $k= k_{1} = k_{2}$. By Definition \ref{def_Thom}, we see that
\[
\xymatrix@C=0.5cm{
  H^{0} (X; \mathcal{O}^\ast_{1} \otimes \mathcal{O}_{2}) \ar[r]^-{u_{1} \cup}_-{\cong} & H^{k} (X, Y; \mathcal{O}_{2}) & \ar[l]_-{u_{2} \cup}^-{\cong} H^{0} (X; \mathbb{Z}) \cong \mathbb{Z} },
\]
where $\mathcal{O}^{*}_{1} := \hom(\mathcal O_1,\Bbb Z)$ is the linear dual of $\mathcal{O}_{1}$, and we use the fact that $H^{k} (X,Y; \mathcal{O}_{1} \otimes \mathcal{O}^{*}_{1} \otimes \mathcal{O}_{2}) = H^{k} (X, Y; \mathcal{O}_{2})$. We claim that $\mathcal{O}_{1} \cong \mathcal{O}_{2}$. For if not, then $\mathcal{O}^\ast_{1} \otimes \mathcal{O}_{2}$ is not constant. Since $X$ is path-connected, we infer that $H^{0} (X; \mathcal{O}^\ast_{1} \otimes \mathcal{O}_{2}) = 0$, which leads to a contradiction.

Since $\mathcal{O}_{1}$ and $\mathcal{O}_{2}$ are systems of infinite cyclic groups, by the path connectivity of $X$ again, there are exactly two bundle isomorphisms $\Phi_{1}$ and $\Phi_{2}$ between $\mathcal{O}_{1}$ and $\mathcal{O}_{2}$ such that $\Phi_{1} = - \Phi_{2}$. We then obtain $\Phi_{1}^{*} = - \Phi_{2}^{*}$. Since each homomorphism $u_{i} \cup$ yields an isomorphism
\[
\mathbb{Z} \cong H^{0} (X; \mathbb{Z}) \overset{\cong}{\longrightarrow} H^{k} (X,Y; \mathcal{O}_{i}),
\]
we infer that $u_{i}$ is a generator of $H^{k} (X,Y; \mathcal{O}_{i}) \cong \mathbb{Z}$. Consequently, we may suitably choose the desired $\Phi$ as $\Phi_{1}$ or $\Phi_{2}$.
\end{proof}

In what follows suppose $X$ is a path-connected space which has a universal cover, $\mathcal{O}$ is an orientation system on $X$, and $u \in H^{k} (X, Y; \mathcal{O})$ is a cohomology class. Let $\mathbf{\Lambda}$ be the local system associated with $\Lambda: = \mathbb{Z} [\pi_{1} (X)]$ on $X$.

\begin{theorem}[Thom Isomorphism: Equivalence]\label{thm_T_eqv}
With respect to the above assumptions the following statements are equivalent.
\begin{enumerate}
\item The homomorphism
\[
\cap u\: H_{k+*} (X, Y; \mathbf{\Lambda}) \overset{\cong}{\longrightarrow} H_{*} (X; \mathbf{\Lambda} \otimes \mathcal{O})\,
\]
is an isomorphism in all degrees.

\item For all local systems $\mathcal{B}$ on $X$, the homomorphism
\[
\cap u\: H_{k+*} (X, Y; \mathcal{B}) \overset{\cong}{\longrightarrow} H_{*} (X; \mathcal{B} \otimes \mathcal{O}).
\]
is an isomorphism in all degrees.

\item For all local systems $\mathcal{B}$ on $X$, the homomorphism
\[
u \cup\: H^{*} (X; \mathcal{B}) \overset{\cong}{\longrightarrow} H^{k+*} (X, Y; \mathcal{O} \otimes \mathcal{B})
\]
is an isomorphism in all degrees.
\end{enumerate}
\end{theorem}
\begin{proof}
Choose a cocycle $\hat{u} \in C^{k} (X,Y; \mathcal{O})$ representing $u$ and
define a chain map
\[
\begin{array}{rrcl}
T\colon\!  C_{k+*} (X,Y; \mathbf{\Lambda}) & \rightarrow & C_{*} (X; \mathbf{\Lambda} \otimes \mathcal{O}) \\
 \xi & \mapsto & \xi \cap \hat{u}\, .
\end{array}
\]
Then $C_{\bullet} (X; \mathbf{\Lambda} \otimes \mathcal{O})$ and $C_{\bullet} (X; \mathbf{\Lambda})$ are free $\Lambda$-complexes, $C_{\bullet} (X; \mathbf{\Lambda} \otimes \mathcal{O}) \cong C_{\bullet} (X; \mathbf{\Lambda})$, and $T$ is a $\Lambda$-homomorphism.

Furthermore $T$ induces chain maps
\[
\cap \hat{u}\: C_{k+*} (X,Y; \mathcal{B}) \longrightarrow C_{*} (X; \mathcal{B} \otimes \mathcal{O})
\]
and
\[
\hat{u} \cup\:  C^{*} (X; \mathcal{B}) \longrightarrow C^{k+*} (X, Y; \mathcal{B} \otimes \mathcal{O}),
\]
where $\hat{u} \cup$ equals the adjoint map of $T$. The statements to be verified are then a direct consequence of \eqref{equ_singular_complex} and Corollary \ref{cor_complex_isomorphism}.
\end{proof}

In the following, assume $(X,Y)$ is a $2$-connected space pair homotopy equivalent to a CW pair, and $X$ is connected. Let $k \ge 3$ be an integer.

\begin{theorem}[Thom Isomorphism: Existence]\label{thm_T_ext}
The pair $(X,Y)$ satisfies the Thom isomorphism in formal dimension $k$ if and only if the homotopy fiber of $Y \to X$ is $S^{k-1}$.
\end{theorem}

\begin{remark}
Theorem \ref{thm_T_ext} is not really new. When  $X$ is $1$-connected, the theorem amounts to \cite[lem.~I.4.3]{Browder2} and \cite[thm.~B]{Klein1} (each proved by different methods). When $X$ is finitely dominated but not necessarily $1$-connected, a proof is sketched in \cite[prop.~3.10]{Ranicki}. For the convenience of the reader, we shall present in Appendix \ref{apd_Thom} a detailed proof in the general case.
\end{remark}

\section{Poincar\'{e} Duality}\label{sec_dual}
In this section, we prove Theorems \ref{thm_interior_duality}, \ref{thm_P_eqv} and \ref{thm_P_crt}, and Corollaries \ref{cor_Wall} and \ref{cor_pair}.

For all but the proof of Theorem  \ref{thm_interior_duality}, the key tool will be the following result which translates between
Poincar\'e duality and the Thom isomorphism.

\begin{proposition}[Duality and the Thom Isomorphism]\label{prop_dual_Thom}
Let $(M; N_{1}, N_{2})$ be a Poincar\'{e} triad with dualizing system $\mathcal{O}_{1}$ and fundamental class
\[
\eta \in H_{n+k} (M, N_{1} \cup N_{2}; \mathcal{O}_{1}).
\]
Suppose $\mathcal{O}_{2}$ is an orientation system on $M$ and $u \in H^{k} (M, N_{1}; \mathcal{O}_{2})$. Let $\mathcal{O}_{3} = \mathcal{O}_{1} \otimes \mathcal{O}_{2}$ and $\xi = \eta \cap u \in H_{n} (M, N_{2}; \mathcal{O}_{3})$.

Then the following statements are equivalent.
\begin{enumerate}
\item The following are isomorphisms for all $*$ and all local systems $\mathcal{B}$:
\[
u \cup: \ \ H^{*} (M; \mathcal{B}) \overset{\cong}{\longrightarrow} H^{k+*} (M, N_{1}; \mathcal{O}_{2} \otimes \mathcal{B}).
\]

\item The following are isomorphisms for all $*$ and all local systems $\mathcal{B}$:
\[
\xi \cap: \ \ H^{*} (M; \mathcal{B}) \overset{\cong}{\longrightarrow} H_{n-*} (M, N_{2}; \mathcal{O}_{3} \otimes \mathcal{B}).
\]

\item The following are isomorphisms for all $*$ and all local systems $\mathcal{B}$:
\[
\cap u: \ \ H_{k+*} (M, N_{1}; \mathcal{B}) \overset{\cong}{\longrightarrow} H_{*} (M; \mathcal{B} \otimes \mathcal{O}_{2}).
\]

\item The following are isomorphisms for all $*$ and all local systems $\mathcal{B}$:
\[
\xi \cap: \ \ H^{*} (M, N_{2}; \mathcal{B}) \overset{\cong}{\longrightarrow} H_{n-*} (M; \mathcal{O}_{3} \otimes \mathcal{B}).
\]
\end{enumerate}
\end{proposition}

\begin{proof}
By Theorem \ref{thm_T_eqv}, we already have the equivalence $(1) \Leftrightarrow (3)$. It suffices to prove $(1) \Leftrightarrow (2)$ and $(3) \Leftrightarrow (4)$.

The following diagram commutes:
\[
\xymatrix{
  H^{*} (M; \mathcal{B}) \ar[dr]_{\xi \cap} \ar[r]^-{u \cup}
                & H^{k+*} (M, N_{1}; \mathcal{O}_{2} \otimes \mathcal{B}) \ar[d]^{\eta \cap}  \\
                & H_{n-*} (M, N_{2}; \mathcal{O}_{3} \otimes \mathcal{B}) \, .            }
\]
Here we have used the fact that
\[
\eta \cap (u \cup a) = (\eta \cap u) \cap a = \xi \cap a.
\]
Since $(M; N_{1}, N_{2})$ is a Poincar\'{e} triad, we know that $\eta \cap$ is an isomorphism, which implies $(1) \Leftrightarrow (2)$.

It remains to prove $(3) \Leftrightarrow (4)$. Consider the following commutative diagram.
\[
\xymatrix{
  H^{*} (M, N_{2}; \mathcal{B}) \ar[dr]_{T} \ar[r]^-{\eta \cap}
                & H_{n+k-*} (M, N_{1}; \mathcal{O}_{1} \otimes \mathcal{B}) \ar[d]^{\cap u}  \\
                & H_{n-*} (M; \mathcal{O}_{1} \otimes \mathcal{B} \otimes \mathcal{O}_{2}) \, ,            }
\]
in which $T(a) := (\eta \cap a) \cap u$. By the skew-commutativity of cup products (Proposition \ref{prop_cup}), we infer
\[
(\eta \cap a) \cap u = \eta \cap (a \cup u) = \pm \eta \cap (u \cup a) = \pm (\eta \cap u) \cap a = \pm \xi \cap a,
\]
where $a \cup u \in H^{k+*} (M, N_{1} \cup N_{2}; \mathcal{B} \otimes \mathcal{O}_{2})$ and $u \cup a \in H^{k+*} (M, N_{1} \cup N_{2}; \mathcal{O}_{2} \otimes \mathcal{B})$. We have identified $\mathcal{B} \otimes \mathcal{O}_{2}$ with $\mathcal{O}_{2} \otimes \mathcal{B}$.

Consequently, $T$ is an isomorphism if and only if
\[
\xi \cap\: H^{*} (M, N_{2}; \mathcal{B}) \rightarrow H_{n-*} (M; \mathcal{O}_{3} \otimes \mathcal{B})
\]
is an isomorphism. Since $\eta \cap$ is an isomorphism, we infer that $\xi \cap$ is an isomorphism if and only if
\[
\cap u\: H_{n+k-*} (M, N_{1}; \mathcal{O}_{1} \otimes \mathcal{B}) \rightarrow H_{n-*} (M; \mathcal{O}_{1} \otimes \mathcal{B} \otimes \mathcal{O}_{2})
\]
is an isomorphism. For a local system $\mathcal{A}$, there is a canonical isomorphism  $\mathcal{A} \cong \mathcal{O}_{1} \otimes (\mathcal{O}_{1}^{*} \otimes \mathcal{A})$. Thus, when $\mathcal{B}$ runs over all local systems, so does $\mathcal{O}_{1} \otimes \mathcal{B}$. Hence, $(3) \Leftrightarrow (4)$.
\end{proof}

\subsection{Digression on thickenings}
The proof of Theorem \ref{thm_P_eqv} will make use of thickenings in the sense of Wall \cite{Wall66}.
Let $\mathbb{H}^{j}$ denote the upper half space in $\Bbb R^j$. This is a manifold with boundary $\Bbb R^{j-1}$.
By an  {\it $m$-thickening} of a finite CW pair $(X,Y)$, we mean a manifold triad $(M;N_1,N_2)$
and a homotopy equivalence of pairs
$$
(X,Y) \overset{\sim}\to (M,N_2)\, ,
$$
where,
\begin{enumerate}
\item $M$ is a compact $m$-manifold with boundary $\partial M$;
\item $\partial M =  N_1 \cup_{\partial} N_2$ a codimension one splitting, i.e.,
$N_1$ and $N_2$ are codimension zero submanifolds of $\partial M$ with common boundary $\partial N_1 = \partial = \partial N_2$.
\end{enumerate}

Let $\Bbb H^m := [0, +\infty) \times \mathbb{R}^{m-1}$ be the upper
half space in $\Bbb R^m$. Assume that $m$ is sufficiently large with respect to the dimension of $X$.
Then
there exists an $m$-thickening $(M;N_1,N_2)$ of $(X,Y)$ in which
$M \subset \Bbb H^m$
and $\partial M \cap \Bbb R^{m-1} = N_2$, where
the latter intersection is transversal (cf.~\cite{Wall66} in the absolute case $Y=\emptyset$; the relative case similarly follows by first 
constructing a thickening $Y$ inside $\Bbb R^{m-1}$ and then thickening inside $\Bbb R^m$ the cells in $X$ that are attached to $Y$).

\begin{proof}[Proof of Theorem \ref{thm_P_eqv}]
We know that $(X,Y)$ is a finitely dominated pair. Taking the product with $S^{3}$ if necessary, we may assume $(X,Y)$ is a finite CW pair (cf.~ Lemma \ref{lem_rd_space} and Lemma \ref{lem_rd_dual}).

As above, there exists an $(n+k)$-thickening $(M;N_1;N_2)$ of $(X,Y)$ with $M \subset \mathbb{H}^{n+k}$, provided that $k$ is sufficiently large.
In particular, we have a homotopy equivalence of pairs
\[
(X,Y) \overset\sim\to (M, N_{2}).
\]
The orientation system $\mathcal{O}$ on $X$ corresponds to an orientation system $\mathcal{O}_{3}$ on $M$, and $[X] \in H_{n} (X,Y; \mathcal{O})$ corresponds to a homology class $\xi \in H_{n} (M, N_{2}; \mathcal{O}_{3})$. Clearly, $(M; N_{1}, N_{2})$ is a Poincar\'{e} triad. We can therefore construct $\mathcal{O}_{1}$, $\mathcal{O}_{2}$, $\eta$ and $u$ as those in Proposition \ref{prop_dual_Thom}. So the conclusion follows from that proposition.
\end{proof}

\begin{remark} In contrast with Theorem \ref{thm_T_eqv},
the proof we gave of Theorem \ref{thm_P_eqv} is not purely algebraic:
we relied on the existence of thickenings of finite CW pairs,
a technique arising in manifold theory. Note that Poincar\'{e} duality  for a manifold triad $(M; N_{1}, N_{2})$ plays a key role in the proof though we just refer to it in one sentence.
\end{remark}

Wall's Conjecture (Corollary \ref{cor_Wall}) and Corollary \ref{cor_pair} will easily follow from Theorem \ref{thm_P_eqv} and the following result.

\begin{lemma}\label{lem_boundary1}
Suppose $\mathcal{O}$ is an orientation system on $X$ and
$[X] \in H_{n} (X,Y; \mathcal{O})$ is a homology class. Suppose \eqref{def_Poincare_pair_1} and \eqref{def_Poincare_pair_2} hold. Then the following homomorphisms are isomorphisms for all $m$ and all local systems $\mathcal{B}$ on $X$:
\[
\partial_{*} [X] \cap\: H^{m} (Y; \mathcal{B}|_{Y}) \overset{\cong}{\longrightarrow} H_{n-1-m} (Y; \mathcal{O}|_{Y} \otimes \mathcal{B}|_{Y}).
\]
\end{lemma}
\begin{proof}
Consider the homology and cohomology long exact sequences of $(X,Y)$ with coefficients $\mathcal{B}$ and $\mathcal{O} \otimes \mathcal{B}$. Then the conclusion follows from the Five Lemma.
\end{proof}

\begin{proof}[Proof of Corollary \ref{cor_pair}]
This follows directly from Theorem \ref{thm_P_eqv} and Lemma \ref{lem_boundary1}.
\end{proof}

\begin{lemma}\label{lem_dual_orientation}
Suppose $(A,B)$ is a path-connected Poincar\'{e} pair (in which $B$ may be empty) having formal dimension $d$ and dualizing system $\mathcal{O}$. Suppose $\mathcal{O}_{1}$ is another orientation system on $A$ such that $H_{d} (A,B; \mathcal{O}_{1}) \neq 0$. Then there is an isomorphism $\mathcal{O}_{1} \cong \mathcal{O}$.
\end{lemma}
\begin{proof}
The proof is similar to that of Theorem \ref{thm_T_unq}. By the duality, we have
\[
H^{0} (A; \mathcal{O}^{*} \otimes \mathcal{O}_{1}) \cong H_{d} (A,B; \mathcal{O} \otimes \mathcal{O}^{*} \otimes \mathcal{O}_{1}) = H_{d} (A,B; \mathcal{O}_{1}) \neq 0.
\]
If $\mathcal{O}_{1} \ncong \mathcal{O}$, then $H^{0} (A; \mathcal{O}^{*} \otimes \mathcal{O}_{1}) = 0$, which leads to a contradiction.
\end{proof}

\begin{lemma}\label{lem_boundary2}
Let $(X, Y)$ be a space pair such that $Y$ is a Poincar\'{e} space of formal dimension $n-1$. Suppose $\mathcal{O}$ is an orientation system on $X$ and $[X] \in H_{n} (X,Y; \mathcal{O})$ is a homology class. Suppose \eqref{def_Poincare_pair_1} and \eqref{def_Poincare_pair_2} hold. Then $(X,Y)$ is a Poincar\'{e} pair with dualizing system $\mathcal{O}$ and  fundamental class $[X]$.
\end{lemma}
\begin{proof}
By Lemma \ref{lem_boundary1}, we obtain the following isomorphism
\[
\partial_{*} [X] \cap\: H^{0} (Y; \mathbb{Z}) \overset{\cong}{\longrightarrow} H_{n-1} (Y; \mathcal{O}|_{Y}).
\]
Therefore, for each path-component $Z$ of $Y$, we have $H_{n-1} (Z; \mathcal{O}|_{Z}) \neq 0$. By Lemma \ref{lem_dual_orientation}, we infer that $\mathcal{O}|_{Y}$ is the dualizing system of $Y$. Then $\partial_{*} [X]$ is a fundamental class of $Y$.
\end{proof}

\begin{proof}[Proof of Corollary \ref{cor_Wall}] This
follows from Theorem \ref{thm_P_eqv} and Lemma \ref{lem_boundary2}.
\end{proof}

We now turn to the proof of Theorem \ref{thm_P_crt}.
\medskip

Suppose $(X,Y_{2})$ is a CW pair satisfying the Thom isomorphism in formal dimension $k$. Similar to the proof of Theorem \ref{thm_T_ext} (cf. Appendix \ref{apd_Thom}), convert $(X,Y_{2})$ to a pair of fibrations
\[
 (CE, E) \rightarrow X,
\]
in which $(CE, E) \simeq (X,Y_{2})$. The fiber over $x \in X$ is identified with the pair $(CF_{x}, F_{x})$, where $CF_{x}$ is contractible. Suppose $u \in H^{k} (CE, E; \mathcal{O}_{2})$ is a Thom class.

\begin{lemma}\label{lem_fiber_Thom}
For each $x \in X$, the restriction of $u$
\[
u|_{CF_{x}} \in H^{k} (CF_{x}, F_{x}; \mathcal{O}_{2,x}) \cong \mathbb{Z}
\]
is a generator of $H^{k} (CF_{x}, F_{x}; \mathcal{O}_{2,x})$.
\end{lemma}
\begin{proof}
We outline two proofs.

In the first instance, by the proof of Lemma \ref{lem_Thom_sufficient}, one can construct a Thom class $u'$ such that $u'|_{CF_{x}}$ is a generator. Then the conclusion follows from the uniqueness of the Thom class (Theorem \ref{thm_T_unq}).

An alternative proof can be derived from the proof Lemma \ref{lem_Thom_necessary}. An argument using the Serre spectral sequence implies the conclusion. For details, see \cite[lem.~I.4.3]{Browder2}.
\end{proof}

\begin{proof}[Proof of Theorem \ref{thm_P_crt}]
(1). Suppose $\eta \in H_{n+k} (X, Y_{1} \cup Y_{2}; \mathcal{O}_{1})$ is a fundamental class of $(X, Y_{1} \cup Y_{2})$. If $F \simeq S^{k-1}$, by Theorem \ref{thm_T_ext}, we know $(X, Y_{1})$ satisfies Thom isomorphism with Thom class $u \in H^{k} (X, Y_{1}; \mathcal{O}_{2})$. Define $\mathcal{O} = \mathcal{O}_{1} \otimes \mathcal{O}_{2}$ and
\[
[X] = \eta \cap u \in H_{n} (X, Y_{2}; \mathcal{O}).
\]
By Proposition \ref{prop_dual_Thom}, we obtain \eqref{def_Poincare_pair_1} and (\ref{def_Poincare_pair_2}). Conversely, if \eqref{def_Poincare_pair_1} or \eqref{def_Poincare_pair_2} holds, then by Proposition \ref{prop_dual_Thom} again, we see that  $(X, Y_{1})$ satisfies Thom isomorphism of formal dimension $k$. By Theorem \ref{thm_T_ext} again, $F \simeq S^{k-1}$.
\medskip

\noindent (2). By Theorem \ref{thm_T_ext}, we know that $F$ and $F'$ are homotopy $(k-1)$-spheres if and only if $(X, Y_{1})$ and $(Y_{2}, Y_{0})$ satisfy the Thom isomorphism in formal dimension $k$. By Proposition \ref{prop_dual_Thom} and Lemma \ref{lem_boundary2}, this is in turn if and only if $(X, Y_{2})$ is a Poincar\'{e} pair of formal dimension $n$.
\medskip

\noindent (3). We know that $F$ and $F'$ are homotopy $(k-1)$-spheres, and $CF$ and $CF'$ are contractible. It suffices to show that the inclusion $(CF', F') \to (CF, F)$ induces an isomorphism
\[
H_{k} (CF', F'; \mathbb{Z}) \overset{\cong}{\longrightarrow} H_{k} (CF, F; \mathbb{Z}).
\]
Let $u \in H^{k} (X, Y_{1}; \mathcal{O}_{2})$ be a Thom class of $(X,Y_{1})$. By Lemma \ref{lem_fiber_Thom}, it suffices to show that $u|_{Y_{2}}$ is a Thom class for $(Y_{2}, Y_{0})$. (Note that, in Lemma \ref{lem_fiber_Thom}, $\mathcal{O}_{2}|_{CF_{x}}$ is actually constant.)

As in (1), let $\eta$ be a fundamental class for $(X, Y_{1} \cup Y_{2})$. Then $\eta \cap u$ is a fundamental class of $(X, Y_{2})$. Then
\[
\partial (\eta \cap u) = \partial \eta \cap u|_{Y_{1} \cup Y_{2}} =  (\partial \eta)|_{Y_{2}} \cap u|_{Y_{2}}
\]
is a fundamental class of the $(n-1)$-dimensional Poincar\'{e} space $Y_{2}$
(here, $(\partial \eta)|_{Y_{2}}$ is defined via excision). Since $(\partial \eta)|_{Y_{2}}$ is a fundamental class of the $(n+k-1)$-dimensional Poincar\'{e} pair $(Y_{2}, Y_{0})$, by Proposition \ref{prop_dual_Thom} again, $u|_{Y_{2}}$ is a Thom class of $(Y_{2}, Y_{0})$.
\end{proof}

Before giving the proof of Theorem  \ref{thm_interior_duality}, we consider a related problem. If $(X,Y)$ and $[X] \in H_{n} (X,Y; \mathbb{Z})$ satisfy \eqref{def_Poincare_pair_1}, \eqref{def_Poincare_pair_2} and \eqref{def_Poincare_pair_3} with integer coefficients, then we say $(X,Y)$ satisfies \textit{Poincar\'{e} duality with integer coefficients}, where the formal dimension is $n$.

When  $Y = \emptyset$, Browder \cite[p.\ 192]{Browder3} exhibited a space $X$ satisfying Poincar\'{e} duality with integer coefficients but which is not a Poincar\'e space. Browder constructed $X$ by taking the wedge of infinitely many acyclic spaces whose fundamental groups are nontrivial (where acyclic in this case means the reduced homology with integer coefficients vanishes). The resulting space clearly satisfies Poincar\'{e} duality with integer coefficients in formal dimension $0$. However, its fundamental group is not finitely generated. Thus, by \cite[cor.~1]{Browder3}, $X$ cannot be a Poincar\'{e} space. Note that Browder's counterexample is not homotopy finite, nor even finitely dominated. We present an alternative class of examples which are actually finite complexes.

\begin{proposition}\label{prop_Z_duality}
For each $n \geq 0$, there are finite CW complexes $X$ which satisfy Poincar\'{e} duality with integer coefficients having formal dimension $n$ but are not Poincar\'{e} spaces of any formal dimension.
\end{proposition}
\begin{proof}
Let $M$ be the Poincar\'{e} homology 3-sphere (see \cite[p.~150-151]{Greenberg_Harper}). 
Then $G = \pi_1 (M)$ is nontrivial, finite and perfect (recall that the latter means $G$ equals its commutator subgroup). 
Moreover, $M$ is a compact smooth $3$-manifold, so it admits the structure of a finite three dimensional CW complex that arises from a triangulation.
Let $A$ be the effect of removing any 3-cell from $M$. Then $A$ is acyclic and  satisfies Poincar\'{e} duality with integer coefficients, where the formal dimension is $0$. We shall prove that $A$ is not a Poincar\'{e} space of any formal dimension.

Since $\pi_{1} (A) = G$ is perfect, a homomorphism from $\pi_{1} (A)$ to $\mathbb{Z}_{2}$ is trivial. Consequently, an orientation system on $A$ must be trivial. If $A$ is a Poincar\'{e} space of formal dimension $n$, then $H_{n} (A; \mathbb{Z})$ is nontrivial, and we infer $n=0$. By Wall's classification \cite[thm. 4.2 (i)]{Wall1}, this is impossible. For the reader's convenience, we give some of the details here. Let $\tilde A\to A$ be the universal cover. Since $\pi_{1} (A)$ is finite, it follows from Theorem \ref{thm_covering} that $\tilde A$ is a Poincar\'{e} space of formal dimension $0$.  Hence, $\tilde A$ is contractible and $A$ is an Eilenberg-Mac~Lane space of type $K(G,1)$. However, since $G$ is finite and nontrivial, $A$ cannot have the homotopy type of a finite  complex (cf.~\cite[prop.~2.45]{Hatcher}), which is a contradiction.

Finally, for any finite complex $P$ satisfying Poincar\'{e} duality with integer coefficients having formal dimension $n$, define $X = A \times P$. Then $X$ is finite and satisfies Poincar\'{e} duality with integer coefficients and has formal dimension $n$. However, $X$ is not a Poincar\'{e} space. This is easily deduced from Theorem \ref{thm_fibration}, but it also follows from the simpler result \cite[thm.~2.5]{Wall1}.
\end{proof}

\begin{proof}[Proof of Theorem  \ref{thm_interior_duality}]
The space $A$ in this theorem is the one in the proof of Proposition \ref{prop_Z_duality}. Furthermore, $Y = A \times S^{n-1}$, and $X = CY$ is the cone of $Y$. Then $H_{*} (Y; \mathbb{Z}) \cong H_{*} (S^{n-1}; \mathbb{Z})$ and $X$ is contractible. Hence, $(X,Y)$ satisfies \eqref{def_Poincare_pair_1} and \eqref{def_Poincare_pair_2} with integer coefficients. Since $X$ is $1$-connected, by the finiteness of $(X,Y)$ and Lemma \ref{lem_rd_dual}, we infer $(X,Y)$ satisfies \eqref{def_Poincare_pair_1} and \eqref{def_Poincare_pair_2}. However, Proposition \ref{prop_Z_duality} shows that $Y$ is not a Poincar\'{e} space.
\end{proof}

\begin{remark}\label{rmk_P_crt}
Suppose $(K,L)$ is a finite CW pair satisfying the first two types of isomorphisms in Definition \ref{def_Poincare_pair}. By taking a thickening of $(K,L)$, we obtain a manifold triad $(X; Y_{1}, Y_{2})$ such that $(X,Y_{2}) \simeq (K,L)$. By assertion (1) of Theorem \ref{thm_P_crt}, the homotopy fiber $F$ of $Y_{1} \to X$ has the homotopy type of $S^{k-1}$. However, Theorem \ref{thm_interior_duality} suggests that Poincar\'e duality might fail for  $Y_{2}$. This motivates the consideration of $F'$ in Assertion (2) of Theorem \ref{thm_P_crt}.
\end{remark}

\section{Doubling}\label{sec_double}
The goal of this section is to prove Theorem \ref{thm_doubling}. The implication ``$\Leftarrow$" of Theorem \ref{thm_doubling} can be easily deduced from the following gluing lemma which was proved by Wall in \cite[thm.~2.1, add.]{Wall1} and reproved in \cite[prop.~C.6]{Goodwillie_Klein}.

\begin{lemma}[Gluing]\label{lem_gluing}
Suppose $(X_{+}; Y_{1}, Y_{2})$ and $(X_{-}; Y_{3}, Y_{2})$ are Poincar\'{e} CW triads of formal dimension $n$. Assume $Y_{1} \cap Y_{2} = Y_{3} \cap Y_{2} = Y_{0}$. Then
\[
(X_{+} \cup_{Y_{2}} X_{-}; Y_{1}, Y_{3})
\] is a Poincar\'{e} triad of formal dimension $n$, where the fundamental classes of $(X_{+}; Y_{1}, Y_{2})$ and $(X_{-}; Y_{3}, Y_{2})$ are the restrictions of those of $(X_{+} \cup_{Y_{2}} X_{-}, Y_{1} \cup_{Y_{0}} Y_{3})$.
\end{lemma}

\begin{proof}[Proof of Theorem \ref{thm_doubling}]
It remains to prove ``$\Rightarrow$".

If necessary, we take the product with $S^{3}$ and use Lemma \ref{lem_rd_space} to obtain a homotopy equivalence of triads
\[
f\: (K; L_{1}, L_{2}) \rightarrow (X; Y_{1}, Y_{2}),
\]
where $(K; L_{1}, L_{2})$ is a finite CW triad. Define $L_{0} = L_{1} \cap L_{2}$. We see that $(K \cup_{L_{2}} K, L_{1} \cup_{L_{0}} L_{1})$ is homotopy equivalent to $(X \cup_{Y_{2}} X, Y_{1} \cup_{Y_{0}} Y_{1})$. By Lemma \ref{lem_rd_dual}, to prove the conclusion, we may assume $(X; Y_{1}, Y_{2})$ is a finite CW triad. We may also assume $X$ is connected.

We first prove the special case $Y_{1} = \emptyset$. Call the latter the
absolute case. Choose a manifold thickening $(M_{+}, N_{2})$ of $(X,Y_2)$ in some $(\mathbb{H}^{m}, \mathbb{R}^{m-1})$, where $\mathbb{H}^{m} = [0, +\infty) \times \mathbb{R}^{m-1}$,
\[
\partial M_{+} = N_{1} \cup_{N_{0}} N_{2}\, ,
\]
$N_{2} = \partial M_{+} \cap \mathbb{R}^{m-1}$, $N_{1}$ is the closure of $\partial M_{+} \setminus N_{2}$, and $N_{0} = N_{1} \cap N_{2}$.  Now glue two copies of $M_{+}$ together along $N_{2}$ to produce a manifold $M$ with boundary $\partial M = N_{1} \cup_{N_{0}} N_{1}$. Clearly, $M$ is a manifold thickening of $X \cup_{Y_{2}} X$. If we choose $m$ sufficiently large, the pairs $(M,\partial M)$ and $(M_{+}, N_{1})$ and $(N_{2}, N_{0})$ will be 2-connected. Since $X \cup_{Y_{2}} X$ and $Y_{2}$ are Poincar\'{e} spaces,  by (2) in Theorem \ref{thm_P_crt}, we know that the homotopy fibers of $\partial M \to M$ and $N_{0} \to N_{2}$ have the homotopy type of $S^{k}$ for some $k \geq 2$. Let $F$ denote the homotopy fiber of the map $\partial M \to M$ taken at a point of $N_{2}$ and let $F_{+}$ be the homotopy fiber of $N_{1} \to M_{+}$ taken at the same point. By Theorem \ref{thm_P_crt} again, it suffices to show that $F_{+}$ is homotopy equivalent to $S^{k}$.

Clearly, since $(M,\partial M)$ is the double of $(M_{+}, N_{1})$, one has an evident retraction $r\colon\! (M, \partial M) \to (M_{+}, N_{1})$ to the inclusion. The map $r$ induces in turn a retraction of homotopy fibers $F \to F_{+}$. But $F$ has the homotopy type of $S^k$. This will imply that the space $F_{+}$ is either contractible or it has the homotopy type of $S^k$. If $F_{+}$ were contractible, then the cohomology groups $H^{*} (M_{+}, N_{1}; \mathcal{B})$ would be trivial in all degrees $*$ and all local systems $\mathcal{B}$ on $M_{+}$. By Poincar\'e duality, it follows that the homology groups $H_{*} (M_{+}, N_{2}; \mathcal{B})$ are also trivial. But $(M_{+}, N_{2}) \simeq (X,Y_2)$, so the groups $H_{*} (X,Y_2;\mathcal B)$ are trivial for all local systems $\mathcal{B}$ on $X$, which implies the isomorphism
\[
H_{*} (Y_{2}; \mathcal{B}) \overset{\cong}{\longrightarrow} H_{*} (X; \mathcal{B})\, .
\]
Since $Y_{2}$ has formal dimension $n-1$, the Mayer-Vietoris exact sequence shows that $H_{n}( X \cup_{Y_{2}} X; \mathcal A) = 0$ for any local system $\mathcal{A}$ on $X \cup_{Y_2} X$. This contradicts the assumption that $X \cup_{Y_{2}} X$ has formal dimension $n$. Consequently, $F_{+}$ is not contractible, so it must be homotopy equivalent to $S^k$, as was to be shown.

We now turn to the case when $Y_{1} \neq \emptyset$. The idea will be to
reduce to the absolute case.  By assumption, $Y_1 \cup_{Y_0} Y_1$ and $Y_0$ are Poincar\'e spaces. By the absolute case, we infer that $(Y_1,Y_0)$ is a Poincar\'e pair. Since $(Y_{2}, Y_{0})$ is also a Poincar\'{e} pair, by Lemma \ref{lem_gluing}, $Y:= Y_1 \cup_{Y_0} Y_2$ is a Poincar\'e space (of formal dimension $n-1$). We also know by assumption that $(X\cup_{Y_2}X,Y_1\cup_{Y_0}Y_1)$ is a Poincar\'e pair. Consider the Poincar\'e pair $(Z,\partial Z)$, where $Z$ is $Y_1 \times [0,1]$ with $Y_0 \times [0,1]$ collapsed to $Y_0$ by means of the projection, and $\partial Z = Y_1\cup_{Y_0}Y_1 $. Then by Lemma \ref{lem_gluing} again, we can glue $Z$ to $X \cup_{Y_{2}} X $ along $\partial Z$ to  obtain the Poincar\'e space
\[
(X\cup_{Y_2}X) \cup_{\partial Z}  Z\, .
\]
Clearly, the latter is homotopy equivalent to $X\cup_{Y} X$. Consequently, $X\cup_{Y} X$ is also a Poincar\'e space. By the absolute case again, we infer that $(X,Y)$ is a Poincar\'e pair. This completes the argument.
\end{proof}

It is natural to ask what happens in Theorem \ref{thm_doubling} if one removes the assumption that
$(Y_{2}, Y_{1} \cap Y_{2})$ is a Poincar\'{e} pair of formal dimension $n-1$. In this case, one can show that the conclusion of the theorem is no longer true:

\begin{example}\label{exm_double}
Let $(X,Y)$ be the pair constructed in Theorem  \ref{thm_interior_duality} with $n > 1$. Then $X \cup_{Y} X$
 is the unreduced suspension of $Y$. Since $Y$ is homologically $S^{n-1}$, we infer that $X \cup_{Y} X \simeq S^{n}$,  which is a Poincar\'{e} space. However, $(X,Y)$ is not a Poincar\'{e} pair.
\end{example}

In another direction, following the proof of Theorem \ref{thm_doubling}, one obtains the following result, which uses Assertion (1) in Theorem \ref{thm_P_crt}.

\begin{proposition}\label{prop_doubling}
Suppose $(X,Y)$ is a finitely dominated pair with $Y \neq \emptyset$. Suppose $X \cup_{Y} X$ is a Poincar\'{e} space of formal dimension $n$. Suppose further $H_{n} (Y; \mathcal{O}|_{Y}) = 0$ for all orientation systems $\mathcal{O}$ on $X \cup_{Y} X$. Then \eqref{def_Poincare_pair_1} and \eqref{def_Poincare_pair_2} hold.
\end{proposition}

Note that Example \ref{exm_double} satisfies the assumptions of Proposition of \ref{prop_doubling}.

\section{Finite Coverings}\label{sec_cover}
The purpose of this section is to prove Theorem \ref{thm_covering}.

Suppose $p\colon\! \bar X \to X$ is covering space, then a local system $\mathcal A$ on $X$ pulls back to a local system $p^\ast\mathcal A$ on $\bar X$. The operation $\mathcal A \mapsto p^\ast \mathcal A$ defines a functor $p^\ast$ from the category of local systems on $X$ to the category of local systems on $\bar X$.

The functor $p^\ast$ has a right adjoint $p_\ast$ defined as follows: if $\mathcal B$ is a local system on $\bar X$, then $p_\ast \mathcal{B}$ is defined by
\[
(p_\ast\mathcal B)_{x} :=  \prod_{y \in p^{-1}(x)} \mathcal B_{y}\, , \qquad x \in X\, .
\]
The fact that this defines a functor uses the path  lifting property.

If we use the direct sum instead of the direct product in the above display, we obtain another local system that will be denoted $p_{!}\mathcal B$. The assignment $\mathcal B \mapsto p_! \mathcal B$ defines a functor $p_!$ which is left adjoint to $p^\ast$. There is also an evident natural transformation
\[
p_{!} \to p_{\ast}
\]
which is an isomorphism whenever $p$ is a finite cover. The functors $p_!$ and $p_\ast$ correspond respectively to induction and coinduction if interpret local systems as modules over the fundamental groupoid. There are natural isomorphisms of chain and cochain complexes
\begin{equation} \label{eqn:identify}
C_{\bullet} (\bar X; \mathcal{B}) \cong C_{\bullet} (X; p_!\mathcal{B})\ \ \text{and} \ \ C^{\bullet} (\bar X;\mathcal B) \cong C^{\bullet} (X; p_\ast\mathcal{B})\, .
\end{equation}
Similarly, if $(X,Y)$ is a pair, and $\bar Y \to Y$ is the restriction of $p$ to $Y$, then we obtain a covering space of pairs $(\bar X,\bar Y) \to (X,Y)$ inducing isomorphisms of relative chain complexes
\begin{equation} \label{eqn:identify2}
C_{\bullet} (\bar X,\bar Y; \mathcal{B}) \cong C_{\bullet} (X,Y; p_!\mathcal{B})\ \ \text{and} \ \ C^{\bullet} (\bar X,\bar Y;\mathcal B) \cong C^{\bullet} (X,Y; p_\ast\mathcal{B})\, .
\end{equation}

\begin{proof}[Proof of Theorem \ref{thm_covering}]
(1). Let $\mathcal{B}$ be an arbitrary local system on $\bar X$. Since $p$ is a finite cover, we have $p_!\mathcal B=p_\ast\mathcal B$.

Suppose $\mu \in C_{n} (X, Y; \mathcal{O})$ is a cycle representing a fundamental class of $(X, Y)$. Let $\bar{\mu} \in C_{n} (\bar X, \bar Y; p^{*} \mathcal{O})$ be the transfer of $\mu$. Then $\bar{\mu}$ is also a cycle. Using the isomorphisms
\eqref{eqn:identify} and \eqref {eqn:identify2} it is straightforward to check that the square
\[
\xymatrix{
  C^{*} (\bar X; \mathcal{B}) \ar[d]_{\bar{\mu} \cap} \ar[r]^-{\cong} & C^{*} (X; p_!\mathcal{B}) \ar[d]^{\mu \cap} \\
  C_{n-*} (\bar X, \bar Y; p^{*} \mathcal{O} \otimes \mathcal{B}) \ar[r]_-{\cong} & C_{n-*} (X, Y; \mathcal{O} \otimes p_! \mathcal{B})   }
\]
commutes. Since $(X,Y)$ is a Poincar\'{e} pair, the right vertical map of
the square is a quasi-isomorphism. Hence the left map is also a quasi-isomorphism.
Similar reasoning shows that the map
\[
\bar{\mu} \cap\: C^{*} (\bar X, \bar Y; \mathcal{B}) \rightarrow C_{n-*} (\bar X; p^{*} \mathcal{O} \otimes \mathcal{B})
\]
is a quasi-isomorphism as well as the map
\[
\partial_{*} \bar{\mu} \cap\: C^{*} (\bar Y; \mathcal{G}) \to  C_{n-1-*} (\bar Y; p^{*} \mathcal{O} \otimes \mathcal{G})
\]
for all local systems $\mathcal{G}$ on $\bar Y$. Consequently $(\bar X, \bar Y)$ is a Poincar\'{e} pair with a fundamental class $[\bar{\mu}]$.
\medskip

\noindent (2). We know that $(X,Y)$ is a finitely dominated pair. By Lemmas \ref{lem_rd_space} and \ref{lem_rd_dual}, we may assume that $(X,Y)$ and $(\bar X, \bar Y)$ are finite CW pairs. By choosing a manifold thickening of $(X,Y)$, we obtain a compact manifold triad $(M; N_{1}, N_{2})$ such that $(X,Y) \simeq (M, N_{2})$. Furthermore, we can assume that $(M, N_{1})$ and $(N_{2}, N_{0})$ are $2$-connected, where $N_{0} = N_{1} \cap N_{2}$. Corresponding to the covering
space pair $(\bar X, \bar Y) \rightarrow (X,Y)$, we obtain a covering space $4$-ad
\[
(\bar M; \bar N_{1}, \bar N_{2}; \bar N_{0}) \rightarrow (M; N_{1}, N_{2}; N_{0}),
\]
in which $(\bar X, \bar Y) \simeq (\bar M; \bar N_{2})$.

By Part (2) of Theorem \ref{thm_P_crt}, $(M, N_{2})$ or $(\bar M; \bar N_{2})$ is a Poincar\'{e} pair if and only if certain homotopy fibers are spheres. Since the inclusions $\bar N_{1} \to \bar M$ (resp. $\bar N_{0} \to \bar N_{2}$) and $N_{1} \to M$ (resp. $N_{0} \to N_{2}$) share the same homotopy fiber, the conclusion holds.
\end{proof}

\section{Fibrations}\label{sec_fibration}
In this section, we shall prove Theorem \ref{thm_fibration}. The strategy of the proof is to reduce to the absolute version using the Doubling Theorem \ref{thm_doubling}.

\begin{proof}[Proof of Theorem \ref{thm_fibration}]
If necessary, we can take product with $S^{3}$ and use Lemmas \ref{lem_rd_space} and \ref{lem_rd_dual}, to assume $(B, \partial B)$, $(F, \partial F)$ and $(E; \partial_{1} E, \partial_{2} E)$ are homotopy finite.

By assumption, $B$ is connected. In particular, the fibers over all points of $B$ have the same homotopy type. We may also assume that $E$ is connected since we can argue component-wise. Note that even when $E$ is connected, the fiber $F$ can be disconnected. However, by the finiteness assumption it only has finitely many components. By choosing a suitable finite covering space $\bar{B} \rightarrow B$, we obtain a commutative diagram
\[
\xymatrix{
 F_{0} \ar[r] & E \ar[dr] \ar[r]
                & \bar{B} \ar[d]  \\
            &    & B            }
\]
where the top row is also a fibration sequence whose fiber $F_{0}$ is a component of $F$.  Using Theorem \ref{thm_covering}, we  are reduced to considering $F_{0} \to E \rightarrow \bar{B}$. This reduces us to considering the case when $F$ is connected and $\bar{B} = B$.

In the absolute case $\partial B = \emptyset$ (hence, $\partial_{2} E = \emptyset$) and $\partial F = \emptyset$ (hence, $\partial_{1} E = \emptyset$), the conclusion was proved in \cite{Gottlieb} and \cite[cor.~F]{Klein2}. (Note that \cite{Gottlieb} and \cite{Klein2} also assumed that $F$ is connected. Furthermore, Part (2) is not stated in either \cite{Gottlieb} or \cite{Klein2}, but this statement is not hard to deduce from their proofs.)

Now suppose $\partial B \neq \emptyset$ and $\partial F = \emptyset$. We apply Theorem \ref{thm_doubling} to the double
\begin{equation}\label{thm_fibration_1}
E \cup_{\partial_{2} E} E \to  B \cup_{\partial B} B\, .
\end{equation}
We first show that \eqref{thm_fibration_1} is a fibration. Indeed, let $r\: B \cup_{\partial B} B \rightarrow B$ be the evident retract. Then $E \cup_{\partial_{2} E} E \to B \cup_{\partial B} B$ is identified with the pullback of the fibration $E \to B$ via $r$. Consequently,  \eqref{thm_fibration_1} is a fibration.

Suppose $(B, \partial B)$ is a Poincar\'{e} pair of formal dimension $n$ and $F$ is a Poincar\'{e} space of formal dimension $k$. Since
\[
F \rightarrow \partial_{2} E \rightarrow \partial B
\]
is a fibration, by the absolute version of the result, $\partial_{2} E$ is a Poincar\'{e} space of formal dimension $n+k-1$. (Note that $\partial B$
and hence $\partial_{2} E$ may be disconnected. If that is the case, we argue component-wise. Note also that all fibers of $\partial_{2} E \rightarrow \partial B$ share the same homotopy type since they are induced from $E \rightarrow B$.) By Theorem \ref{thm_doubling}, we see that $B \cup_{\partial B} B$ is a Poincar\'{e} space of formal dimension $n$. By the absolute version again, $E \cup_{\partial_{2} E} E$ is a Poincar\'{e} space of formal dimension $n+k$. By Theorem \ref{thm_doubling} again, $(E, \partial_{2} E)$ is a Poincar\'{e} pair of formal dimension $n+k$.

Conversely, suppose $(E, \partial_{2} E)$ is a Poincar\'{e} pair. Applying the absolute version to $\partial_{2} E$, we infer that $\partial B$ and $F$ are Poincar\'{e} spaces of formal dimension $n-1$ and $k$ respectively for some $n$ and $k$. By Theorem \ref{thm_doubling}, $E \cup_{\partial_{2} E} E$ is a Poincar\'{e} space of formal dimension $n+k$. By the absolute version again, $B \cup_{\partial B} B$ is a Poincar\'{e} space of formal dimension $n$. By Theorem \ref{thm_doubling} again, $(B, \partial B)$ is a Poincar\'{e} pair of formal dimension $n$. This establishes the result when $\partial F = \emptyset$.

It remains to prove the result when $\partial F \neq \emptyset$. Consider the map
\begin{equation}\label{thm_fibration_2}
 E \cup_{\partial_{1} E} E \to B.
\end{equation}
Since $(E, \partial_{1} E) \rightarrow B$ is a fibration pair and $\partial_{1} E \rightarrow E$ is a closed cofibration, by \cite[prop.~1.3]{Clapp}, the map \eqref{thm_fibration_2} is a fibration whose fiber over the basepoint is $F \cup_{\partial F} F$. The conclusion now follows by applying the result in the previous case using Theorem \ref{thm_doubling} applied to the triad $(E; \partial_{1} E, \partial_{2} E)$.
\end{proof}

\begin{remark}
Besides Theorem \ref{thm_fibration}, the absolute fibration theorem has some other extensions. For example, a similar argument shows that under suitable finiteness assumptions $(E; E|_{\partial_{1} B}, E|_{\partial_{2} B})$ is a Poincar\'{e} triad if and only if $(B; \partial_{1} B, \partial_{2} B)$ is a Poincar\'{e} triad and $F$ is a Poincar\'{e} space.
\end{remark}

\begin{remark}
For a fibration $F \to E \to B$, if we do not assume that $F$ is finitely dominated, then the conclusion of Theorem \ref{thm_fibration} will no longer be true. For example, consider the path fibration $\Omega B \to PB \to B$. Then $PB$ is contractible and hence a Poincar\'{e} space of formal dimension $0$. However, $\Omega B$ is usually not a Poincar\'{e} space even if $B$ is a finite Poincar\'{e} complex. The example of the double cover $S^\infty \rightarrow \mathbb{R} P^\infty$ shows that the finiteness assumption on $B$ is also needed for the result to hold.
\end{remark}

\section{Historical Remarks}\label{sec_history}
Here we describe various inequivalent definitions of Poincar\'{e} pairs appearing in the literature. The difference among them is captured in the statement of Theorem \ref{thm_P_crt}. Historically, there were at least four different definitions, all of which are equivalent in the case of finite CW pairs $(X,Y)$ in which $X$ and $Y$ are $1$-connected.

The earliest and simplest definition was provided by Browder \cite{Browder0}, \cite[p.~927]{Browder1}, which only assumes \eqref{def_Poincare_pair_1}, \eqref{def_Poincare_pair_2} and (\ref{def_Poincare_pair_3}) with integer coefficients, where the dualizing system is necessarily trivial. Browder's definition suffices for doing surgery theory in the simply connected case (c.f.~\cite{Browder2}). However, as seen in Proposition \ref{prop_Z_duality}, it is inadequate more generally.

Wall gave two definitions. The definition in \cite[p.\ 215]{Wall1} is equivalent to the one used in this paper. In \cite[sec.~2]{Wall2}, Wall defined the notion of a {\it simple} Poincar\'{e} pair. This definition takes simple homotopy type into account, and is well-adapted for use in surgery theory in the non-simply connected case.

Yet another definition is due to Spivak \cite[p.\ 82]{Spivak}. Suppose $(X,Y)$ is a finite CW pair and $p: \tilde X \rightarrow X$ is a universal cover. Denote by $\tilde Y = p^{-1} (Y)$. Spivak considered homology with locally finite chains, denoted by $H_{*}^{\text{LF}}$ (which coincides with Borel-Moore homology). His definition respectively exchanges \eqref{def_Poincare_pair_1}, \eqref{def_Poincare_pair_2} and (\ref{def_Poincare_pair_3}) with isomorphisms
\[
[\tilde X] \cap\: H^{m} (\tilde X; \mathbb{Z}) \overset{\cong}{\longrightarrow} H_{n-m}^{LF} (\tilde X, \tilde Y; \mathbb{Z}),
\]
\[
[\tilde X] \cap\: H^{m} (\tilde X, \tilde Y; \mathbb{Z}) \overset{\cong}{\longrightarrow} H_{n-m}^{LF} (\tilde X; \mathbb{Z}),
\]
and
\[
\partial_{*} [\tilde X] \cap\: H^{m} (\tilde Y; \mathbb{Z}) \overset{\cong}{\longrightarrow} H_{n-1-m}^{LF} (\tilde Y; \mathbb{Z}).
\]

There are obvious differences between Spivak's definition and Wall's \cite{Wall1}.  In particular, if $X$ is 1-connected, then Spivak's definition amounts to Browder's definition.  However, if $Y$ isn't $1$-connected, then $(X,Y)$ may fail to satisfy Wall's definition. A case in point is the $(X,Y)$ in Theorem \ref{thm_interior_duality}. Note that, with merely integer coefficients, \eqref{def_Poincare_pair_1} and  \eqref{def_Poincare_pair_2} imply \eqref{def_Poincare_pair_3} (cf.~Lemma \ref{lem_boundary1}). Consequently, $(X,Y)$ satisfies Spivak's definition but does not satisfy Wall's.

Lastly, we relate Part (2) of Theorem \ref{thm_P_crt} to the above, which is similar to Spivak's \cite[prop.~4.6]{Spivak}. In the latter, Spivak ignored the homotopy fiber $F'$: he claimed that $(X,Y_{2})$ is a Poincar\'{e} pair if and only if $F$ itself is a homotopy sphere. In fact, if we suppose $F$ is a homotopy sphere, it is not hard to show that $(X,Y_{2})$ is indeed a Poincar\'{e} pair with Spivak's definition. However, if $Y_{2}$ is not $1$-connected, Remark \ref{rmk_P_crt} shows that $(X,Y_{2})$ may not satisfy Wall's definition.

\appendix

\section{Skew-Commutativity of Cup Products}
In this section, we consider the cup product of cohomology with local systems and prove its skew-commutativity Proposition \ref{prop_cup}. In what follows $R$ is a commutative ring with unit.

Let $\mathcal{G}$ and $\mathcal{H}$ be local systems of $R$-modules on a space $A$. Then we have a natural isomorphism
\begin{equation}\label{equ_tensor_swap}
\begin{array}{rcl}
\mathcal{H} \otimes_{R} \mathcal{G} & \longrightarrow & \mathcal{G} \otimes_{R} \mathcal{H} \, ,\\
h \otimes g & \mapsto & g \otimes h\, .
\end{array}
\end{equation}
\begin{definition}\label{def_excisive_couple}
Let $B_{1}, B_{2} \subseteq A$ be subspaces of $A$. One says that $\{ B_{1}, B_{2} \}$ is  an {\it excisive couple} if the evident homomorphisms
\[
H_{*} (B_{1}, B_{1} \cap B_{2}; \mathcal{B}) \longrightarrow H_{*} (B_{1} \cup B_{2}, B_{2}; \mathcal{B})
\]
and
\[
H_{*} (B_{2}, B_{1} \cap B_{2}; \mathcal{B}) \longrightarrow H_{*} (B_{1} \cup B_{2}, B_{1}; \mathcal{B})
\]
are isomorphisms for all $*$ and all local systems $\mathcal{B}$ on $B_{1} \cup B_{2}$.
\end{definition}

\begin{example}
If the interiors of $B_{1}$ and $B_{2}$ with respect to the topology of $B_{1} \cup B_{2}$ cover $B_{1} \cup B_{2}$, then $\{ B_{1}, B_{2} \}$ is an excisive couple.
\end{example}

\begin{example}
Suppose both $B_{1}$ and $B_{2}$ are closed in $B_{1} \cup B_{2}$. If the inclusions  $B_{1} \cap B_{2} \to B_{1}$ and $B_{1} \cap B_{2} \to B_{2}$ are cofibrations, then $\{ B_{1}, B_{2} \}$ is an excisive couple.
\end{example}

\begin{proposition}[Skew-Commutativity]\label{prop_cup}
Assume $\{ B_{1}, B_{2} \}$ is an excisive couple. Then for every $a \in H^{p} (A, B_{1}; \mathcal{G})$ and $b \in H^{q} (A, B_{2}; \mathcal{H})$, one has
\[
a \cup b = (-1)^{pq} \varphi (b \cup a) \in H^{p+q} (A, B_{1} \cup B_{2}; \mathcal{G} \otimes_{R} \mathcal{H}),
\]
where
\[
\varphi\colon\! \ \ H^{p+q} (A, B_{1} \cup B_{2}; \mathcal{H} \otimes_{R} \mathcal{G}) \longrightarrow H^{p+q} (A, B_{1} \cup B_{2}; \mathcal{G} \otimes_{R} \mathcal{H})
\]
is the natural isomorphism induced by \eqref{equ_tensor_swap}.
\end{proposition}

Suppose $\mathcal{B}$ be the local system on $A$. Let
\[
S_{p} = [e_{0}, e_{1}, \cdots,e_{p}]
\]
be the standard $p$-simplex. Consider the affine map $\theta_p: S_{p} \to S_{p}$ defined on vertices by
\[
\theta_{p} (e_{i}) = e_{p-i}\, .
\]
Then the collection $\{\theta_p\}_{p \ge 0}$ defines a chain map
\[
\theta \colon\!  C_{\bullet} (A; \mathcal{B}) \longrightarrow C_{\bullet} (A; \mathcal{B})\, ,
\]
where
\[
\theta (\sigma) := (-1)^{\frac{1}{2} p(p+1)} (\sigma \circ {\theta}_{p}), \qquad \sigma \colon\! S_{p} \to A\, .
\]
Clearly, $\theta$ is an involution.

\begin{lemma}
The chain map $\theta$ on $C_{\bullet} (A; \mathcal{B}) \longrightarrow C_{\bullet} (A; \mathcal{B})$ is chain homotopic to the identity, i.e. there is $R$-module homomorphism
\[
J: \ \ C_{\bullet} (A; \mathcal{B}) \longrightarrow C_{\bullet} (A; \mathcal{B})
\]
which raises the degree by $1$ such that
\[
\mathrm{id} - \theta = \partial J + J \partial.
\]
Furthermore, if $B \subseteq A$, then $J (C_{\bullet} (B; \mathcal{B})) \subseteq C_{\bullet} (B; \mathcal{B})$.
\end{lemma}

\begin{proof} We use acyclic models.
The proof is classical if $\mathcal{B}$ is
the constant local system $\mathbb{Z}$ (see e.g. \cite[(24.11)]{Greenberg_Harper}),
and the argument is essentially the same in the general case.

For a singular simplex $\sigma: S_{p} \rightarrow A$, denote by $C (\sigma)$ the free abelian group generated by $\sigma \circ \tau$, where
\[
\tau \colon\! S_{q} \rightarrow S_{p}
\]
ranges over all affine maps the send vertices to vertices. Then $C (\sigma)$ is a free $\mathbb{Z}$-chain complex, and $\mathrm{id} - \theta$ is carried by $C (\sigma)$. In other words, there exists $\hat{J}: C (\sigma) \rightarrow C (\sigma)$ such that
\[
\mathrm{id} - \theta = \partial \hat{J} + \hat{J} \partial.
\]
Denote by $x_{0} = \sigma (e_{0})$. Then the desired $J$ can then be defined as
\[
J = 1 \otimes \hat{J}\: \mathcal{B}_{x_{0}} \otimes C (\sigma) \rightarrow \mathcal{B}_{x_{0}} \otimes  C (\sigma)\, . \qedhere
\]
\end{proof}

The last lemma immediately implies the following.

\begin{lemma}\label{lem_permutation}
Suppose $B_{1} \subseteq A$ and $B_{2} \subseteq A$. Then the chain map defined above determines a chain endomorphism of the quotient complex
\[
C_{\bullet} (A; \mathcal{B}) / \left( C_{\bullet} (B_{1}; \mathcal{B}) +  C_{\bullet} (B_{2}; \mathcal{B}) \right)
\]
which we also denote by $\theta$. Furthermore, the latter is chain homotopic to the identity.
\end{lemma}

\begin{proof}[Proof of Proposition \ref{prop_cup}]
The conclusion follows immediately from Lemma \ref{lem_permutation}. The argument is similar to \cite[p.\ 201-202]{Greenberg_Harper}.
\end{proof}

\section{The K\"{u}nneth Theorems}
In this section we describe  the K\"{u}nneth Theorems for homology and cohomology with local systems. One would have expected such results to be classical, and in fact they were used by Wall in \cite[thm.~2.5]{Wall1}. However, Wall neither formulated the results, nor proved them, nor gave a reference. To the best of our knowledge, a paper by Greenblatt \cite[thms.~1.6,~1.7]{Greenblatt} is the only one on this subject. Unfortunately, \cite{Greenblatt} is insufficient for our purposes (as it didn't deal with space pairs and secondly, it didn't characterize explicitly the homology cross product). The purpose of this section is to fill a gap in the literature.

Define the product of space pairs as
\[
(A, B) \times (C, D) = (A \times C, (A \times D) \cup (B \times C)).
\]
Then we can consider whether $\{A \times D, B \times C\}$ is an excisive couple (see Definition \ref{def_excisive_couple}).

\begin{example}
If $B = \emptyset$ or $D = \emptyset$, then $\{ A \times D, B \times C \}$ is an excisive couple.
\end{example}

\begin{example}
If both $B$ and $D$ are open, then $\{ A \times D, B \times C \}$ is an excisive couple.
\end{example}

\begin{example}\label{exm_excisive}
If both $B \to A$ and $D \to C$ are closed cofibrations, then $\{ A \times D, B \times C \}$ is an excisive couple.
\end{example}

In what follows, let $R$ be a principal ideal domain. Suppose $(A,B)$ and $(C,D)$ are space pairs such that $\{ A \times D, B \times C \}$ is an excisive couple, and $A$ and $C$ are path-connected and have universal covers. Suppose $\mathcal{G}$ is a local system of \textbf{free} $R$-modules on $A$, and $\mathcal{H}$ is a local system of $R$-modules on $C$.

\begin{theorem}[K\"{u}nneth Theorem for Homology]\label{thm_Kunneth_h}
With respect to the above assumptions we have, for all $n \ge 0$, a splitting exact sequence
\[
{\small
\xymatrix{
  0 \ar[r] & \underset{q+r=n}{\bigoplus} H_{q} (A, B; \mathcal{G}) {\otimes_{R}} H_{r} (C, D; \mathcal{H}) \ar[r]^-{\times} &
  H_{n} ((A,B) {\times} (C,D); \mathcal{G} {\otimes_{R} }\mathcal{H}) \\
    \ar[r] & \underset{q+r=n-1}{\bigoplus} \mathrm{Tor}_{R} (H_{q} (A, B; \mathcal{G}), H_{r} (C, D; \mathcal{H})) \ar[r] & 0 \, , }}
\]
where the first displayed homomorphism is the homological cross product (cf.\ right before the proof of Theorem \ref{thm_Kunneth_cross}).
\end{theorem}

Define the cohomological cross product by
\[
\times\colon\!  H^{q} (A, B; \mathcal{G}) \otimes_{R} H^{r} (C, D; \mathcal{H}) \to
H^{q+r} ((A,B) {\times} (C,D); \mathcal{G} {\otimes_{R} }\mathcal{H})
\]
by the formula
\[
a \times b := p_{1}^{*} (a) \cup p_{2}^{*} (b)\, ,
\]
where $p_{1}$ and $p_{2}$ are the projections of $A \times C$ onto $A$ and $C$ respectively. Recall that a map $(K,L) \rightarrow (A,B)$ is a CW approximation of $(A,B)$ if the map is a weak homotopy equivalence of pairs (for the definition, see \cite[p.~219]{G.Whitehead}) and $(K,L)$ is a CW pair.

\begin{theorem}[K\"{u}nneth Theorem for Cohomology]\label{thm_Kunneth_c}
With respect to  the assumptions of Theorem \ref{thm_Kunneth_h}, assume further $(A,B)$ can be approximated by a CW pair $(K,L)$ such that $K \setminus L$ contains only finitely many cells in each dimension. In addition, assume one of the following two conditions:
\begin{enumerate}
\item $\mathcal{G}$ is a local system of finitely generated $R$-modules;

\item $(C,D)$ can be approximated by a CW pair $(P,Q)$ such that $P \setminus Q$ contains only finitely many cells in each dimension.
\end{enumerate}

Then for all $n$ there is a splitting exact sequence
\[
\xymatrix{
  0 \ar[r] & \underset{q+r=n}{\bigoplus} H^{q} (A, B; \mathcal{G}){\otimes_{R}} H^{r} (C, D; \mathcal{H}) \ar[r]^-{\times} & H^{n} ((A,B) {\times} (C,D);
{\mathcal{G} \otimes_{R} \mathcal{H}}) \\
   \ar[r] & \underset{q+r=n+1}{\bigoplus} \mathrm{Tor}_{R} (H^{q} (A, B; \mathcal{G}), H^{r} (C, D; \mathcal{H})) \ar[r] & 0 \, . }
\]
\end{theorem}

\begin{theorem}\label{thm_Kunneth_cross}
The cross product in Theorem \ref{thm_Kunneth_h} may be characterized as follows. Given $q_{1} \leq q$, $r_{1} \leq r$, $\xi \in H_{q} (A, B; \mathcal{G})$, $\eta \in H_{r} (C, D; \mathcal{H})$, $a \in H^{q_{1}} (A; \mathcal{G}_{1})$, and $b \in H^{r_{1}} (C; \mathcal{H}_{1})$. Then
\[
(\xi \times \eta) \cap (a \times b) = (-1)^{(q-q_{1})r_{1}} (\xi \cap a) \times (\eta \cap b),
\]
i.e. the following diagram commutes, where tensoring is over $R$:
\[
{\small
\xymatrix{
  H_{q} (A, B; \mathcal{G}) \otimes H_{r} (C, D; \mathcal{H}) \ar[d]_{(-1)^{(q-q_{1})r_{1}} (\cap a) \otimes (\cap b)} \ar[r]^-{\times} & H_{q+r} (*; \mathcal{G} \otimes \mathcal{H}) \ar[d]^{\cap (a \times b)} \\
  H_{q-q_{1}} (A, B; \mathcal{G} \otimes \mathcal{G}_{1}) \otimes H_{r-r_{1}} (C, D; \mathcal{H} \otimes \mathcal{H}_{1}) \ar[r]_-{\times} & H_{ q -q_{1} + r - r_{1} } (*; \mathcal{G} \otimes \mathcal{H} \otimes \mathcal{G}_{1} \otimes \mathcal{H}_{1})  }
  }
\]
Here $*$ stands for $(A,B) \times (C,D)$, and $\mathcal{G} \otimes \mathcal{H} \otimes \mathcal{G}_{1} \otimes \mathcal{H}_{1}$ is naturally identified with $\mathcal{G} \otimes \mathcal{G}_{1} \otimes \mathcal{H} \otimes \mathcal{H}_{1}$.
\end{theorem}

\begin{proof}[Proof of Theorem \ref{thm_Kunneth_h}]
Consider the universal covers $\theta_{1}: \tilde A \rightarrow A$ and $\theta_{2}: \tilde C \rightarrow C$. Let $\tilde B = \theta_{1}^{-1} (B)$ and $\tilde D = \theta_{2}^{-1} (D)$. Since $\{ A \times D, B \times C \}$ is an excisive couple, by the Eilenberg-Zilber theorem \cite[(29.34)]{Greenberg_Harper}, the Alexander-Whitney chain map
\[
C_{\bullet} (\tilde{A} \times \tilde{C}; R) / \left( C_{\bullet} (\tilde{A} \times \tilde{D}; R) + C_{\bullet} (\tilde{B} \times \tilde{C}; R) \right) \rightarrow C_{\bullet} (\tilde A, \tilde B; R) \otimes C_{\bullet} (\tilde C, \tilde D; R)
\]
and the natural map
\[
C_{\bullet} (\tilde{A} \times \tilde{C}; R) / \left( C_{\bullet} (\tilde{A} \times \tilde{D}; R) + C_{\bullet} (\tilde{B} \times \tilde{C}; R) \right) \rightarrow    C_{\bullet} ((\tilde A, \tilde B) \times (\tilde C, \tilde D); R)
\]
are chain equivalences.

Denote by $G_{1} = \pi_{1} (A)$ and $G_{2} = \pi_{1} (C)$. Then $C_{\bullet} ((\tilde A, \tilde B) \times (\tilde C, \tilde D); R)$ is a free $R[G_{1} \times G_{2}]$-complex and
\[
R[G_{1} \times G_{2}] \cong R[G_{1}] \otimes_{R} R[G_{2}]\, .
\]
Furthermore, $C_{\bullet} (\tilde A, \tilde B; R)$ is a free $R[G_{1}]$-complex, and $C_{\bullet} (\tilde C, \tilde D; R)$ is a free $R[G_{2}]$-complex. In the light of these observations, the above two chain equivalences are $R[G_{1} \times G_{2}]$-homomorphisms. Therefore, these two chain maps induce a chain equivalence of $R$-complexes
\[
C_{\bullet} ((A,B) \times (C,D); \mathcal{G} \otimes_{R} \mathcal{H}) \rightarrow C_{\bullet} (A, B; \mathcal{G}) \otimes_{R} C_{\bullet} (C, D; \mathcal{H}).
\]

Assuming $\mathcal{G}$ is $R$-free, we infer that $C_{\bullet} (A, B; \mathcal{G})$ is a free $R$-complex. Since $R$ is a PID, we obtain the desired splitting exact sequence by applying the algebraic K\"{u}nneth Theorem (cf. \cite[thm.~V.2.1]{Hilton_Stammbach}) to $C_{\bullet} (A, B; \mathcal{G}) \otimes_{R} C_{\bullet} (C, D; \mathcal{H})$. Note that the first displayed homomorphism is induced by a chain homotopy inverse of the Alexander-Whitney chain map, thus it is the cross product (cf. right before the proof of Theorem \ref{thm_Kunneth_cross}).
\end{proof}

\begin{proof}[Proof of Theorem \ref{thm_Kunneth_c}]
We can always choose a CW approximation $(P,Q) \rightarrow (C,D)$ (cf. \cite[p.~219]{G.Whitehead}).
If condition (2) is assumed, then we take $(P,Q)$ as the pair satisfying this condition. For brevity, we also let $\mathcal{G}$ and $\mathcal{H}$ denote the pullback local systems on $K$ and $P$ respectively via the CW approximations.

We have assumed $\{ A \times D, B \times C \}$ is an excisive couple. By Example \ref{exm_excisive}, so is $\{ K \times Q, L \times P \}$. Therefore, by the Eilenberg-Zilber Theorem, the singular cochain complexes $C^{\bullet} ((A,B) \times (C,D); \mathcal{G} \otimes_{R} \mathcal{H})$ and $C^{\bullet} ((K,L) \times (P,Q); \mathcal{G} \otimes_{R} \mathcal{H})$ are chain equivalent to
\[
\mathrm{Hom}_{R[G_{1} \times G_{2}]} \left( C_{\bullet} (\tilde A, \tilde B; R) \otimes_{R} C_{\bullet} (\tilde C, \tilde D; R), M_{1} \otimes_{R} M_{2}  \right),
\]
and
\[
\mathrm{Hom}_{R[G_{1} \times G_{2}]} \left( C_{\bullet} (\tilde K, \tilde L; R) \otimes_{R} C_{\bullet} (\tilde P, \tilde Q; R), M_{1} \otimes_{R} M_{2}  \right)
\]
respectively, where $M_{1}$ (resp.\ $M_{2}$) is the restriction of $\mathcal{G}$ (resp.\ $\mathcal{H}$) to the base point of $A$ (resp.\ $C$). Since $C_{\bullet} (\tilde K, \tilde L; R)$ (resp. $C_{\bullet} (\tilde P, \tilde Q; R)$) is chain equivalent to $C_{\bullet} (\tilde A, \tilde B; R)$ (resp. $C_{\bullet} (\tilde C, \tilde D; R)$), it suffices to obtain the desired exact sequence for $(K,L) \times (P,Q)$. By \cite[lem.~1]{Wall0}, we can further replace the above (co)chain complexes with the corresponding cellular (co)chain complexes by virtue of chain equivalences.

Let $\hat{C}_{\bullet}$ and $\hat{C}^{\bullet}$ denote the cellular chain and cochain complexes respectively. Since $K \setminus L$ has only finitely many cells in each dimension, $\hat{C}_{\bullet} (\tilde K, \tilde L; R)$ is a free $R[G_{1}]$-complex with finite rank in each degree. Moreover, the condition (1) in the assumption means that $M_{1}$ is a finitely generated $R$-module; the condition (2) means that $\hat{C}_{\bullet} (\tilde P, \tilde Q; R)$ is a free $R[G_{2}]$-complex with finite rank in each degeree. Since (1) or (2) is additionally satisfied, we have
\begin{eqnarray*}
&  & \mathrm{Hom}_{R[G_{1} \times G_{2}]} \left( \hat{C}_{\bullet} (\tilde K, \tilde L; R) \otimes_{R} \hat{C}_{\bullet} (\tilde P, \tilde Q; R), M_{1} \otimes_{R} M_{2}  \right) \\
& = & \mathrm{Hom}_{R[G_{1}]} \left( \hat{C}_{\bullet} (\tilde K, \tilde L; R), M_{1} \right) \otimes_{R} \mathrm{Hom}_{R[G_{2}]} \left( \hat{C}_{\bullet} (\tilde P, \tilde Q; R), M_{2} \right) \\
& = & \hat{C}^{\bullet}(K,L; \mathcal{G}) \otimes_{R} \hat{C}^{\bullet}(P,Q; \mathcal{H}).
\end{eqnarray*}
As $\mathcal{G}$ is a system of free $R$-modules, we infer that $\hat{C}^{\bullet}(K,L; \mathcal{G})$ is a free $R$-complex. Now the desired exact sequence follows from the algebraic K\"{u}nneth Theorem. Finally, by the definition of the Alexander-Whitney map, the cross product in this exact sequence is defined by cup products.
\end{proof}

Before proving  \ref{thm_Kunneth_cross}, we recall the definition of the homology cross product. This product actually arises from a canonical triangulation of the product of two simplices.

Denote by $S_{q}$ the standard $q$-simplex. Clearly, the vertices of $S_{q}$ and $S_{r}$ are ordered. Give the vertices of $S_{q} \times S_{r}$ the dictionary order. We denote by $(i,j)$ the vertices of $S_{q} \times S_{r}$, where $0 \leq i \leq q$ and $0 \leq j \leq r$. The notation $(i,j)$ means that the first coordinate is the $i$th vertex of $S_{q}$ and the second one is the $j$th vertex of $S_{r}$. Using the ordering on $S_{q} \times S_{r}$, there is a canonical way to triangulate $S_{q} \times S_{r}$ (see \cite[p.\ 277-278]{Hatcher} and \cite[(29.27)]{Greenberg_Harper}). The simplices of $S_{q} \times S_{r}$ are of the form
\begin{equation}\label{product_trig}
[(i_{0}, j_{0}), (i_{1}, j_{1}), \cdots, (i_{m}, j_{m})]
\end{equation}
such that $0 \leq i_{k} \leq q$, $0 \leq j_{k} \leq r$, $i_{k} \leq i_{k+1}$ and $j_{k} \leq j_{k+1}$ for all $k$, where these $(i_{k}, j_{k})$ are distinct.

Using the triangulation, we see that the projections $p_{1}: S_{q} \times S_{r} \rightarrow S_{q}$ and $p_{2}: S_{q} \times S_{r} \rightarrow S_{r}$ are simplicial maps preserving the ordering of vertices.

The cross product for the simplicial chain complex of $S_{q} \times S_{r}$ is defined as
\begin{equation}\label{cross}
\begin{array}{rcl}
\Delta_q (S_{q}; R) \otimes_{R} \Delta_r (S_{r}; R) & \overset{\times}{\longrightarrow} & \Delta_{q+r}(S_{q} \times S_{r}; R) \\ \empty
[S_{q}] \otimes [S_{r}] & \mapsto & \sum \pm [(i_{0}, j_{0}), (i_{1}, j_{1}), \cdots, (i_{q+r}, j_{q+r})],
\end{array}
\end{equation}
where the sum is over all $(q+r)$-simplices of $S_{q} \times S_{r}$, the signs $\pm$ are suitably chosen as $+$ or $-$ so that the orientations equal the product orientation of $S_{q} \times S_{r}$.

Let $(A,B)$ and $(C,D)$ be finite simplicial pairs. Order the vertices of $A$ and $C$. We can then obtain a simplicial structure on $A \times C$. This simplicial structure results in the cross product on the chain level of simplicial chain complexes:
\[
\Delta_q(A,B; \mathcal{G}) \otimes_{R} \Delta_r(C,D; \mathcal{H}) \overset{\times}\to \Delta_{q+r}( (A,B) \times (C,D); \mathcal{G} \otimes_{R} \mathcal{H})
\]
which is similar to \eqref{cross}.

More generally, suppose $(A,B)$ and $(C,D)$ are pairs of topological spaces. We can similarly define the cross product on the chain level of singular chain complexes, which yields a chain homotopy inverse of the Alexander-Whitney homomorphism (cf. \cite[(29.27)]{Greenberg_Harper}).

\begin{proof}[Proof of Theorem \ref{thm_Kunneth_cross}]
If $(A,B)$ and $(C,D)$ are finite simplicial pairs, the proof is elementary. In this case, we can use instead the simplicial chain and cochain complexes. The conclusion is true even on the chain level.

In general, suppose $(A,B)$ and $(C,D)$ are pairs of topological spaces. As homology classes are compactly supported, there exist finite simplicial pairs $(K,L)$ and $(P,Q)$ and maps
\[
f\colon\!  (K,L) \rightarrow (A,B) \qquad \text{ and } \qquad g\colon\!  (P,Q) \rightarrow (C,D)
\]
satisfying the following: there exist $\xi' \in H_{q} (K,L; f^{*} \mathcal{G})$ and $\eta' \in H_{r} (P,Q; g^{*} \mathcal{H})$ such that
\[
\xi = f_{*} \xi' \qquad \text{and} \qquad \eta = g_{*} \eta'.
\]
Let
\[
a' = f^{*} a \in H^{q_{1}} (K,L; f^{*} \mathcal{G}_{1}) \qquad \text{and} \qquad b' = g^{*} b \in H^{r_{1}} (P,Q; g^{*} \mathcal{H}_{1}).
\]
Since the conclusion holds for $(K,L) \times (P,Q)$, we infer that
\[
(\xi' \times \eta') \cap (a' \times b') = (-1)^{(q-q_{1})r_{1}} (\xi' \cap a') \times (\eta' \cap b').
\]
Applying $(f \times g)_{*}$ to the above equation, we complete the proof.
\end{proof}

\section{Proof of Theorem \ref{thm_T_ext}}\label{apd_Thom}
Since the homotopy fiber and all algebraic objects involved in Theorem \ref{thm_T_ext} are homotopy invariants, in what follows we may assume $(X,Y)$ is a CW pair and $X$ is connected.

Recall the standard way to factor the inclusion $Y \to X$ as
\[
Y \rightarrow E \to X\, ,
\]
where $Y \to E$ is both a cofibration and a homotopy equivalence and $E\to X$ is a fibration.  Here, $E$ is the fiber product
\[
Y \times_X X^I
\]
given by the space of paths $\alpha\colon\!  [0,1] \to X$ satisfying $\alpha(0) \in Y$. The map $Y \to E$ is defined  by mapping a point to the constant path at that point and the map $E\to X$ sends a path to its terminus $\alpha(1)$. Let $CE := X^I$ be the free path space consisting of all paths in $X$.

Then we have an evident fibration pair
\[
(CE, E) \to X
\]
whose fiber at $x\in X$ is the pair $(CF_x,F_x)$, where $F_x$ is the homotopy fiber of the inclusion $Y \to X$ at $x\in X$ and $CF_x$ is the (contractible) space of paths in $X$ that are based at $x$. The evident inclusion $(X,Y) \to (CE,E)$ is a homotopy equivalence.

The proof of Theorem \ref{thm_T_ext} is divided into the following two lemmas.

\begin{lemma}\label{lem_Thom_sufficient}
The implication $``\Leftarrow"$ of Theorem \ref{thm_T_ext} is valid.
\end{lemma}
\begin{proof}
Under the assumption $F \simeq S^{k-1}$, we shall prove $(X,Y)$ satisfies the Thom isomorphism. It suffices to show that $(CE, E)$ satisfies the Thom isomorphism.

The first step is to construct the desired orientation system $\mathcal{O}$ on $X$. For each $x\in X$ we define
\[
\mathcal{O}_{x} := H_{k} (CF_{x}, F_{x}; \mathbb{Z}) .
\]
Note that this is an infinite cyclic group. Let $\gamma: [0,1] \rightarrow X$ be a path such that $\gamma (0) = a$ and $\gamma (1) = b$. Then $\gamma$ induces an isomorphism from $\mathcal{O}_{a}$ to $\mathcal{O}_{b}$. Thus we get the desired orientation system $\mathcal{O}$. For convenience, we also denote by $\mathcal{O}$ the orientation system $p_{1}^{*} \mathcal{O}$ on $CE$.

The second step is to construct a Thom class $u$. Since $X$ is a CW complex, it's locally contractible. Therefore, $(CE, E)$ is locally homotopically a pair of trivial fibrations. By a Mayer-Vietoris argument similar to \cite[thm.~10.2]{milnor_stasheff}, we obtain the following conclusions:
\begin{enumerate}
\item For each finite subcomplex $K$ of $X$, for all local systems $\mathcal{B}$ on $X$ and $*<k$, the group
\[
H_{*} (CE|_{K}, E|_{K}; \mathcal{B})
\]
is trivial.
\item For each finite subcomplex $K$ of $X$, there exists a unique
\[
u_{K} \in H^{k} (CE|_{K}, E|_{K}; \mathcal{O})
\]
such that, $\forall x \in K$, the restriction of $u_{K}$ to $(CF_{x}, F_{x})$ is given by the preferred generator of
 \[
H^{k} (CF_{x}, F_{x}; \mathcal{O}_{x}) = \mathrm{Hom} (H_{k} (CF_{x}, F_{x}; \mathbb{Z}), \mathcal{O}_{x}) = \mathrm{Hom} (\mathcal{O}_{x}, \mathcal{O}_{x}) = \mathbb{Z}\, .
\]

\item If $K_{1} \subseteq K_{2} \subseteq X$ and $K_{i}$ are finite subcomplexes, then
\[
u_{K_{2}}|_{K_{1}} = u_{K_{1}}.
\]
\end{enumerate}
(Note that $(CF_{x}, F_{x})$ is homotopy equivalent to a finite CW pair. We can, therefore, apply the K\"{u}nneth Theorems \ref{thm_Kunneth_h} and \ref{thm_Kunneth_c} in the above argument.)

As homology is compactly supported, we have for any subcomplex $A$ of $X$, a canonical identification
\begin{equation}\label{equ_direct_lim}
H_{*} (CE|_{A}, E|_{A}; \mathcal{B}) = \lim_{\overrightarrow{K \subseteq A}} H_{*} (CE|_{K}, E|_{K}; \mathcal{B}),
\end{equation}
where the direct limit is taken over all finite subcomplexes $K$ of $A$. By (1) above, we infer
\begin{equation}\label{equ_homology_vanish}
H_{*} (CE|_{A}, E|_{A}; \mathcal{B}) = 0
\end{equation}
for all $*<k$. By (3) above, there exists a unique
\[
u \in \lim_{\overleftarrow{K \subseteq X}} H^{k} (CE|_{K}, E|_{K}; \mathcal{O})
\]
such that $u|_{K} = u_{K}$. We claim that
\begin{equation}\label{equ_inverse_lim}
\lim_{\overleftarrow{K \subseteq X}} H^{k} (CE|_{K}, E|_{K}; \mathcal{O}) = H^{k} (CE, E; \mathcal{O})
\end{equation}
which will imply $u \in H^{k} (CE, E; \mathcal{O})$. The verification of (\ref{equ_inverse_lim}) is similar to \cite[thm.~10.4]{milnor_stasheff} (if $\mathcal{O}$ is constant, we can just duplicate that argument). Let $x_{0}$ be the base point of $X$. For a subcomplex $A$ containing $x_{0}$, by the Universal Coefficient Spectral Sequence \cite[thm.~(2.3)]{Levine} (see also the proof of Lemma \ref{lem_acyclic}) and (\ref{equ_homology_vanish}), we obtain
\[
H^{k} (CE|_{A}, E_{A}; \mathcal{O}) = \mathrm{Hom}_{\Lambda} (H_{k} (CE_{A}, E|_{A}; \mathbf{\Lambda}); \mathcal{O}_{x_{0}}).
\]
This together with (\ref{equ_direct_lim}) implies (\ref{equ_inverse_lim}).

It remains to prove the isomorphism in Definition \ref{def_Thom}. By a Mayer-Vietoris argument, we see the isomorphism (2) in Theorem \ref{thm_T_eqv} holds for all $(CE|_{K}, E|_{K})$ where $K$ ranges over the finite subcomplexes of $X$. As homology is compactly supported, we obtain the isomorphism (2) for $(X,Y)$. Consequently, the  implication ``$\Leftarrow$" follows from Theorem \ref{thm_T_eqv}.
\end{proof}

It remains to prove the other direction of Theorem \ref{thm_T_ext}.

\begin{lemma}\label{lem_Thom_necessary}
The implication ``$\Rightarrow$'' of Theorem \ref{thm_T_ext} is valid.
\end{lemma}
\begin{proof}
By the $2$-connectivity of $(X,Y)$, we infer that $F$ is $1$-connected. Since both $X$ and $Y$ are CW complexes, by \cite[cor.~(13)]{Stasheff}, $F$ is homotopy equivalent to a CW complex. Thus it suffices to prove
\begin{equation}\label{equ_homology_sphere}
H_{*} (CF, F; \mathbb{Z}) =
\begin{cases}
\mathbb{Z}, & * =k; \\
0, & * \neq k.
\end{cases}
\end{equation}

Let $p: \tilde X \rightarrow X$ be the universal covering. Let $\tilde Y = p^{-1} (Y)$. We get
\[
p^{*} u \in H^{k} (\tilde X, \tilde Y; p^{*} \mathcal{O}) = H^{k} (\tilde X, \tilde Y; \mathbb{Z}).
\]
By Theorem \ref{thm_T_eqv}, we have the isomorphisms
\[
\cap u\: H_{k+*} (X,Y; \mathbf{\Lambda}) \overset{\cong}{\longrightarrow} H_{*} (X; \mathbf{\Lambda} \otimes \mathcal{O})
\]
which are equivalent to the isomorphisms
\begin{equation}\label{lem_Thom_necessary_1}
\cap p^{*} u\: H_{k+*} (\tilde X, \tilde Y; \mathbb{Z}) \overset{\cong}{\longrightarrow} H_{*} (\tilde X; \mathbb{Z}).
\end{equation}
Note that the commutative square
\[
\xymatrix{
\tilde Y \ar[r] \ar[d] & Y\ar[d]\\
\tilde X \ar[r] & X
}
\]
is a pullback. In particular, the homotopy fiber of the left vertical map may be identified with $F$.
There are two ways to complete the proof.

By (\ref{lem_Thom_necessary_1}) and Theorem \ref{thm_T_eqv}, we obtain isomorphisms in cohomology
\[
p^{*}u\, \cup\: H^{*} (\tilde X; \mathbb{Z}) \overset{\cong}{\longrightarrow} H^{k+*} (\tilde X, \tilde Y; \mathbb{Z}).
\]
Since $\tilde X$ is $1$-connected, the conclusion follows from \cite[lem.~I.4.3]{Browder2}.

Alternatively, apply \cite[thm.~B]{Klein1} or an argument using the Serre spectral sequence to \eqref{lem_Thom_necessary_1} directly. (Note that \cite{Browder2} applies the Serre spectral sequence to cohomology. We apply here a similar argument to the homology.)
\end{proof}


\end{document}